\documentclass[10pt, reqno]{amsart}
\usepackage{graphicx, amssymb, amsmath, amsthm}
\numberwithin{equation}{section}
\usepackage{mathrsfs}
\usepackage{mathtools}
\mathtoolsset{showonlyrefs}
\usepackage[shortlabels]{enumitem}
\usepackage{tikz}
\usepackage{microtype}
\usepackage[T1]{fontenc}
\newcommand{\todo}[1]{}
\usepackage[colorlinks=true%
    ,linkcolor=blue%
    ,citecolor=blue%
    ,urlcolor=blue%
backref=page]{hyperref}
\hypersetup{urlcolor=blue, citecolor=blue}

\renewcommand{\tilde}{\widetilde}
\renewcommand{\hat}{\widehat}
\newcommand\A{\mathbf A}

\DeclareMathOperator*{\spn}{span}
\DeclareMathOperator{\AB}{AB}
\newcommand{\fd}{\mathfrak D}

\newcommand{\myC}{\mathcal C}

\newcommand{\myH}{\mathbb{H}} 
\newcommand{\HH}{\mathcal{H}}
\newcommand{\hankel}{\mathscr H} 

\newcommand{\coloneq}{\coloneqq }

\newcommand\LL{\mathcal{L}}
\newcommand{\II}{\mathcal{I}_{\ell}^{\perp}}

\newcommand{\IL}{\mathcal{I}_\ell}

\newcommand{\DMB}{\mathcal{DMB}}

\let\Re=\undefined\DeclareMathOperator*{\Re}{Re}
\let\Im=\undefined\DeclareMathOperator*{\Im}{Im}

\newcommand{\R}{\mathbb{R}}
\renewcommand{\S}{\mathbb{S}}
\newcommand{\C}{\mathbb{C}}
\DeclareMathOperator{\myc}{c}
\DeclareMathOperator{\mys}{s}

\newcommand{\Z}{\mathbb{Z}}
\newcommand{\N}{\mathbb{N}}
\newcommand{\dd}{\mathop{}\!\mathrm{d}}

\newtheorem{theorem}{Theorem}[section]
\newtheorem{prop}[theorem]{Proposition}
\newtheorem{lemma}[theorem]{Lemma}
\newtheorem{corollary}[theorem]{Corollary}

\newtheorem{proposition}[theorem]{Proposition}
\theoremstyle{definition}
\newtheorem{definition}[theorem]{Definition}
\newtheorem{remark}[theorem]{Remark}

\newtheorem{assumption}[theorem]{Assumption}

\begin{document}
\title{Intertwining operators beyond the Stark Effect}

\author[L. Fanelli]{Luca Fanelli}
\address{Luca Fanelli \newline
Ikerbasque, Basque Foundation for Science; Departamento de Matem\'aticas, Universidad del Pa\'{\i}s Vasco/Euskal Herriko Unibertsitatea
PV/EHU 48940 Leioa, Spain; BCAM - Basque Center for Applied Mathematics 48009 Bilbao, Spain}
\email{lfanelli@bcamath.org}

\author[X. Su]{Xiaoyan Su}
\address{Xiaoyan Su\newline
  Loughborough University, Loughborough, United Kingdom}
\email{X.Su2@lboro.ac.uk}

\author[Y. Wang]{Ying Wang}
\address{Ying Wang\newline
  Basque Center for Applied Mathematics, Bilbao, Spain}
\email{ywang@bcamath.org}

\author[J. Zhang]{Junyong Zhang}
\address{Junyong Zhang\newline
  Department of Mathematics, Beijing Institute of Technology, Beijing 100081}
\email{zhang\_junyong@bit.edu.cn}

\author[J. Zheng]{Jiqiang Zheng}
\address{Jiqiang Zheng \newline
  Institute of Applied Physics and Computational Mathematics  and National Key Laboratory of Computational Physics\\ Beijing 100088\\ China}
\email{zheng\_jiqiang@iapcm.ac.cn}

\begin{abstract}
  The main mathematical manifestation of the Stark effect in quantum mechanics is the shift and the formation of clusters of eigenvalues when a spherical Hamiltonian is perturbed by lower order terms.
  Understanding this mechanism turned out to be fundamental in the description of the large-time asymptotics of the associated Schr\"odinger groups and can be responsible for the lack of dispersion \cite{FFFP1, FFFP, FFFP2}.
  Recently, Miao, Su, and Zheng introduced in \cite{MSZ} a family of spectrally projected \emph{intertwining operators}, reminiscent of the Kato's wave operators, in the case of constant perturbations on the sphere (inverse-square potential), and also proved their boundedness in $L^p$.
  Our aim is to establish a general framework in which some suitable intertwining operators can be defined also for non constant spherical perturbations in space dimensions 2 and higher.
  In addition, we investigate the mapping properties between $L^p$-spaces of these operators.
  In 2D, we prove a complete result, for the Schr\"odinger Hamiltonian with a (fixed) magnetic potential an electric potential, both scaling critical, allowing us to prove dispersive estimates, uniform resolvent estimates, and $L^p$-bounds of Bochner--Riesz means.
  In higher dimensions, apart from recovering the example of inverse-square potential, we can conjecture a complete result in presence of some symmetries (zonal potentials), and open some interesting spectral problems concerning the asymptotics of eigenfunctions.
  \\[1em]
\end{abstract}

\maketitle


\section{Introduction}
In 1913, Johannes Stark observed a peculiar phenomenon now known as the Stark effect: under the influence of a non-constant external electric field, the energy levels of an atom shift and split into multiple lines, clustering into specific formations \cite{Stark}.
This effect is analogous to the mechanism observed under a static external magnetic field, as first identified by Pieter Zeeman in 1896 \cite{Zeeman}.
Mathematical evidence of the above effect can be obtained by analyzing the spectrum of the spherical Schr\"odinger Hamiltonian, $-\Delta_{\mathbb{S}^{d-1}}$ when perturbed by a lower-order external force.
Due to the significant experimental and theoretical implications, extensive literature has been devoted to the study of the spectrum of operators of the form $-\Delta_{\mathbb{S}^{d-1}} + a(\theta)$, where $a:\mathbb{S}^{d-1} \rightarrow \mathbb{R}$.
This type of perturbation typically results in the formation of eigenvalue clusters around the unperturbed energy levels $j(j+d-2)$, with $j \in \mathbb{N}$, together with a uniform shift, provided that the potential is non-constant and satisfies certain symmetry requirements.
Some pioneering works, such as those by D. Gurarie \cite{Gur} and by L. Thomas and S. Wassell \cite{TW}, have established this clustering behavior, with extensive references therein. When $d=2$ (corresponding to a one-dimensional spherical problem), this clustering effect is nearly universal, requiring only minimal regularity conditions on the potentials (see \cite[Lemma 2.1]{FFFP} for detailed asymptotics of eigenvalues and eigenfunctions.)

In the past decade, beginning with \cite{FFFP1}, a deep connection has emerged between these spectral properties and the large-time asymptotics of scaling-invariant Schr\"odinger groups.
Broadly, the structure and separation of high-energy clusters, along with the asymptotic profiles of the associated eigenfunctions, are crucial in analyzing the high-energy band in the time-decay properties of these Schr\"odinger groups in the full space, as showed in several recent papers (see e.g. \cite{FFFP, FFFP2, FGK15, FZZ, FZZ23, GWZZ, GYZZ, MYZ}).
A canonical model within this context is the Schr\"odinger Hamiltonian with an inverse-square potential $-\Delta+\frac{a}{|x|^2}$, $a$ being a constant.
Restricting this model to $\mathbb S^{d-1}$ results in a constant perturbation of the Laplace-Beltrami operator on the sphere, $-\Delta_{\mathbb S^{d-1}}+a$.
The large-time behavior of the propagator $\cramped{e^{it(-\Delta+\frac{a}{|x|^2})}}$, including decay and Strichartz estimates, is now well understood, particularly for dimensions $d=2$ and $3$.
In this case, the spectral problem is straightforward as the spectrum only sees a uniform shift, without cluster formation.
Specifically, the spectrum of $-\Delta_{\mathbb S^{d-1}}+a$ is given by the numbers $j(j+d-2)+a$, with $j\in\mathbb N$, and the eigenfunctions are spherical harmonics, with shifted index.
This leads to a direct one-to-one correspondence between the unperturbed and perturbed eigenvalues, preserving their order and allowing the definition of some operators that map spectral projections before perturbation to their perturbed counterparts.
Miao, Su, and Zheng in \cite{MSZ} introduced such operators, which exhibit an intertwining property that allows to represent the perturbed functional calculus in terms of the unperturbed one.
This concept aligns with the classic wave operators introduced by T. Kato, \cite{K}, whose $L^p$-boundedness properties have been extensively studied, due to their applications in dispersive estimates and scattering theory, in \cite{DF1, EGG, EGL, W, Y1, Y2, Y3, Y4, Y5, Y6} and many other papers.

The objective of this manuscript is to establish a rigorous mathematical foundation for defining similar intertwining operators as those in \cite{MSZ} also in presence of cluster formations.
Extending this analysis beyond the Stark effect needs a precise combinatorial analysis of the energy levels and their associated eigenfunctions after the splitting. Once this structure is understood, we propose a general methodology to prove the $L^p$-boundedness of these operators.
This involves a novel discrete multiplier result with variable coefficients (Lemma \ref{lem:Dismul} below) and detailed decay estimates of oscillatory integrals.

\subsection{Mathematical framework}\label{subsec:framework}

Let  us consider an electromagnetic Schr\"odinger operator on $\R^{d}$, $d \geq 2$, of the form
\begin{equation}\label{equ:lAdef}
  \mathcal{L}_{\A,a}\coloneq\left(-i\nabla+\frac{\A(\theta)}{|x|}\right)^2+\frac{a(\theta)}{|x|^2},
\end{equation}
where $\theta=\frac{x}{|x|}\in\mathbb{S}^{d-1}$, $a\in L^{\infty}(\mathbb{S}^{d-1},\mathbb{R})$ and $\A\in C^1 (\mathbb{S}^{d-1},\R^d)$  is a transversal vector field,
\begin{equation}\label{eq:transversal}
  \A(\theta)\cdot\theta=0\text{ for all }\theta\in\mathbb{S}^{d-1}.
\end{equation}
The differential operator $\mathcal{L}_{\A,a}$ acts on complex-valued functions $f\in C^\infty(\R^d\setminus\{0\})$ as
\begin{align*}
  \mathcal{L}_{\A,a}f =-\Delta f+\frac{|\A(\theta)|^2 +a(\theta)-i \operatorname{div}_{\mathbb S^{d-1}} {\A (\theta)}}{|x|^2} -2i \frac{\A(\theta)}{ |x|} \cdot \nabla f,
\end{align*}
where $\operatorname{div}_{\mathbb S^{d-1}} {\A (\theta)}$ denotes the Riemannian divergence of ${\A }$ on the unit sphere $\mathbb S^{d-1}$ endowed with the standard metric. \todo{speciality of this operator, cause difficulty and some method works}

In spherical coordinates $x=r \theta$, $r=|x|$, $\theta=\frac{x}{|x|}$, $\mathcal{L}_{\A,a}$ can be written as
\begin{equation}\label{LAa-r}
  \begin{split}
    \mathcal{L}_{\A,a}=-\partial_r^2-\frac{d-1}r\partial_r+\frac{L_{\A,a}}{r^2},
  \end{split}
\end{equation}
where the spherical operator  $L_{\A,a}$ is defined by
\begin{align}
  L_{\A,a}=(-i\nabla_{\mathbb{S}^{d-1}}+\A)^{2}+a(\theta),
\end{align}
where $\nabla_{\mathbb{S}^{d-1}}$ is the spherical gradient on the unit sphere $\mathbb{S}^{d-1}$.
By  standard spectral theoretical arguments, we see that the spectrum of $L_{\A,a}$ is purely discrete, and consists of a diverging sequence of real eigenvalues
$\lambda_n(\A,a) \uparrow + \infty $ where each eigenvalue is repeated according to its finite multiplicity (see e.g. \cite{FFT}).

Under the condition that
\begin{align}\label{equ:a-Aassum}
  \lambda_1(\A ,a)\geq-\Big(\frac{d-2}{2}\Big)^2,
\end{align}
the Hamiltonian $\mathcal{L}_{\A ,a}$ can be understood as the Friedrichs' extension of the natural quadratic form on $L^2(\R^d,\C)$ \cite[Theorem VI.2.1]{K}, with the form domain
$$
  \mathcal D(\mathcal L_{\A ,a})
  =\bigg\{
  f\in L^2(\R^d;\C):
  \int_{\R^d}
  \bigg\lvert
  \nabla f+i\frac{\A(\theta)}{|x|}f\bigg\lvert^2+\frac{a(\theta)}{|x|^2}|f|^2
  \dd x
  <\infty
  \bigg\}.
$$
\todo{Definition of the operator}

As explained above, the goal of this manuscript is to introduce some intertwining operators which enable us to reduce the functional calculus associated to $\mathcal L_{\A, a}$ to the one for the electric-free operator $\mathcal L_{\A,0}$. Once such an object is available, its mapping properties between $L^p$-spaces will permit us to obtain relevant information about the large-time asymptotics of the Schr\"odinger group $e^{it\mathcal L_{\A, a}}$, as well as the resolvent of $\mathcal L_{\A, a}$, by inheriting these properties from $\mathcal L_{\A, 0}$.

\todo{special example: AB potential, history and gap}
When $d=2$, a fundamental example of a magnetic field in the above class is the so called \emph{Aharonov--Bohm} potential (hereafter $\AB$), defined by
\begin{equation}\label{ab-potential}
  \tag{Aharonov--Bohm}
  \AB(x):=\frac{\A(\theta)}{|x|},
  \qquad\A(\theta)=\alpha                                                                                           \left(-\frac{x_2}{|x|}, \frac{x_1}{|x|}\right),\quad \alpha\in\R.
\end{equation}
In Aharonov and Bohm's paper \cite{AB}, the phenomenon of scattering in regions in which the electromagnetic field is absent was predicted within the framework of Schr\"odinger dynamics. This phenomenon was later
incontrovertibly confirmed by experiments of Tonomura et al. \cite{PT89}. The $\AB$ model serves as a 2D idealization of a 3D model involving an infinitely long and thin solenoid carrying an electric current, which confines a magnetic field within its interior. Recently, experimental evidence of 2D quasiparticles, known as \emph{anyons}, which naturally carry on an $\AB$-type magnetic field \cite{BKB}, has reinforced the significance of this 2D model. Notably, the $\AB$ Schr\"odinger Hamiltonian $\mathcal L_{\AB,0}$ retains the same scaling invariance as the free operator $\mathcal L_{0,0}$, highlighting the scaling-critical nature of this perturbation.

In the context of zero-order perturbations, the canonical example in any dimension is the inverse-square potential operator $\mathcal L_{0,a}$ where $\cramped{a\geq-\frac{(d-2)^2}{4}}$ is a constant. Also in this case, we have that $\mathcal L_{0,a}$ scales as $\mathcal L_{0,0}$.

Extensive research on such operators has focused on the validity of some families of estimates that characterize the Schr\"odinger equation as the prototypical dispersive PDE. Burq, Planchon, Stalker, and Tahvildar-Zadeh proved in \cite{BPSS, BPST} that the standard Strichartz estimates are valid in the presence of zero-order perturbations with the same scaling as the inverse-square. Later, in \cite{FFFP1, FFFP} it was shown that the decay estimate
\begin{equation}\label{eq:decayshro}
  \|e^{it\LL_{\A,a}}\|_{L^1(\R^2)\to L^\infty(\R^2)}\lesssim |t|^{-1},
\end{equation}
which implies Strichartz estimates by standard results, holds generally in 2D. In 3D, this estimate holds for the defocusing inverse-square potential case $a\geq0$, but fails for the focusing case
$-(d-2)^2/4\leq a<0$ \cite{FFFP1}. The proofs in \cite{FFFP1, FFFP} rely on a detailed analysis of the Stark and Zeeman effects post-perturbation. These results have been later extended to other dispersive operators (see \cite{FZZ, FZZ23, GYZZ, MYZ} and references therein).

This paper aims to address the spectral structure of operators $\mathcal L_{\A, a}$ and $\mathcal L_{\A,0}$ for fixed $\A$, by developing a rigorous framework which may also hold in presence of cluster formations for defining intertwining operators under specific spectral conditions. Establishing such operators and bounding them in $L^p$-spaces will require careful spectral analysis, as well as some novel harmonic analytical tools which will come into play in the sequel.

\subsection{Spectral framework}
The free spherical operator $L_{0, 0}=-\Delta_{\mathbb S^{d-1}}$ has purely discrete spectrum, given by the eigenvalues $\eta_{j}^2$ where $\eta_j\coloneq \sqrt{j(j+d-2)}\ge0$, each one with multiplicity $m_j\coloneq\frac{(d+j-1) !}{(d-1)! j!}=O(j^{d-2})$. Moreover, the usual spherical harmonics are the corresponding eigenfunctions. Having compact inverse, the operators $L_{\A, 0}$ and $L_{\A, a}$ also have purely discrete spectrum with eigenvalues accumulating at infinity (see e.g. \cite{FFFP1}).
The expected spectral picture for $L_{\A,a}$ after the Stark--Zeeman effect suggests to assume from now on that the following conditions hold for $L_{\A, 0}$ and $L_{\A, a}$:
\begin{assumption}\label{eigenvalue assumption}
  \ \par
  \begin{enumerate}
    \item \label{assump1}
          \textbf{Asymptotics of clusters.} There exist $\ell\in\N$ and a constant $0<C_{\A}<1/2$, only depending on $\A$, such that for any $j\geq\ell$, the free eigenvalues $\eta_j^2=j(j+d-2)$  split into clusters of eigenvalues of the operators $L_{\A, 0}$ and $L_{\A, a}$, each with exactly $m_j\coloneq\frac{(d+j-1) !}{(d-1)! j!}$ eigenvalues (with possible repetitions), denoted by $\mu_{jk}^2$ and $\nu_{jk}^2$, $k=1,\dots, m_j$, respectively, with the following localization property:
          \begin{gather*}
            \eta_{j} +C_{\A}-\frac12 < \mu_{j 1} \leq \mu_{j 2} \leq \dots \leq \mu_{j m_j} \leq \eta_{j} +C_{\A}+ \frac12, \\
            \eta_{j} +C_{\A}-\frac12 < \nu_{j 1} \leq \nu_{j 2} \leq \dots \leq \nu_{j m_j} \leq \eta_{j} +C_{\A}+ \frac12.
          \end{gather*}
          The corresponding normalized eigenfunctions are denoted by $e_{jk}$ and $\phi_{jk}$ $( k=1, \dots, m_j )$, respectively, i.e. for all $j \geq \ell,$ and all $1\leq k \leq m_j$,
          \[ L_{\A ,0}e_{jk}(\theta)=\mu_{jk}^2e_{jk}(\theta),
            \quad\text{and}\quad
            L_{\A ,a}\phi_{jk}(\theta)=\nu_{jk}^2\phi_{jk}(\theta).   \]
    \item \label{assump2} \textbf{Simultaneous completion.} There exists $\ell_0$ (possibly zero) depending on $\ell$, a set \(
          \{e_i\}_{i=1,2, \dots, \ell_0}
          \) of normalized eigenfunctions
          of $L_{\A, 0}$, and a set \(
          \{\phi_i\}_{i=1,2, \dots, \ell_0}
          \) of normalized eigenfunctions
          of $L_{\A, a}$,
          with corresponding eigenvalues $\mu_i^2, \nu_i^2 \in [-\frac{(d-2)^2}{4}, \eta_\ell+C_{\A}-\frac{1}{2}]$ respectively, such that
          $\{e_i\}_{i=1,2, \dots, \ell_0}\cup \{e_{jk}\}_{j\geq \ell, k}$ and $\{\phi_i\}_{i=1,2, \dots, \ell_0}\cup \{\phi_{jk}\}_{j\geq \ell, k}$  are both bases of $L^2(\mathbb S^{d-1})$.
  \end{enumerate}
\end{assumption}
In the above, we allow $\mu_i^2,\nu_i^2<0$ for $i<\ell_0$ for notational convenience.
\begin{remark}\label{rem:explain1}
  Assumption \ref{eigenvalue assumption} \eqref{assump1} is describing the typical asymptotic spectral picture after the occurrence of the Stark--Zeeman effect. First we notice a shift of the $\eta_j^2$'s. This is due to the presence of $\A$ and $a$, but asymptotically the dominant order of the shift is given by $\A$ (see e.g. \cite[Lemma 2.1]{FFFP}), and this is the reason why $C_\A$ does not depend on $a$. Then the free eigenvalues split in a cluster of new eigenvalues, and the total multiplicity remains equal to $m_j$. This last fact is generic in dimension 2, and it's known to occur under some suitable symmetry assumptions in higher dimensions.
  Among the examples of potentials satisfying \eqref{assump1}, we list:
  \begin{itemize}
    \item
          for any $d\geq2$, the inverse-square potential, namely $L_{\A,a}$, with $a$ constant;
    \item
          in $d=2$, provided $\A \in W^{1, \infty}(\mathbb{S}^{1},\mathbb{R}^2)$ and $a\in W^{1, \infty}(\mathbb{S}^{1},\mathbb{R})$, see Section \ref{sec: dimension two} for details;
    \item
          in $d\geq 3$, provided $\A=0$ and $a$ enjoys \emph{zonal} symmetry \cite{Gur, TW}.
  \end{itemize}
\end{remark}
\begin{remark}\label{rem:explain2}
  Assumption \ref{eigenvalue assumption} \eqref{assump2} is about the low frequency part of the spectra: under this assumption, we can always complete the above families of eigenfunctions to orthonormal bases of $L^2(\mathbb S^{d-1})$, by adding \emph{exactly} the same number of low frequencies $j\le\ell_0$ in the two cases $L_{\A, 0}$ and $L_{\A, a}$. This is crucial in order to be able to give the definition of intertwining operators in Section \ref{equ:Wjthetadef} below. Although \eqref{assump2} could seem to be easily satisfied, we did not find in the literature any result along these lines. Hence we needed to perform a suitable analysis in order to give examples of potentials which fully satisfy Assumption \ref{eigenvalue assumption}, see Section \ref{sec: dimension two}.
\end{remark}
From now on, we denote the index sets by
\begin{gather*}
  \IL \coloneq \{ (j,k): j\ge \ell, \ 1\le k \le m_j\}, \quad
  \II \coloneq \{ 1,\dots,\ell_0\} ,
\end{gather*}
so that $\{ e_\alpha \}_{\alpha \in \II \cup \IL}$ and $\{ \phi_\alpha \}_{\alpha \in {\II} \cup \IL}$  are two bases of $L^2(\S^{d-1})$; we will also write $\sum_\alpha$ instead of $\sum_{\alpha \in \II\cup\IL}$ and $\sum_{j,k}$ to mean $\sum_{j=\ell}^\infty\sum_{k=1}^{m_j}$ when the intent is clear.

We now aim to introduce the spectral measures associated to $\mathcal L_{\A,0}$ and $\mathcal L_{\A,a}$. We first notice that under  Assumption \ref{eigenvalue assumption},
we have the following two orthogonal decompositions of $L^2(\mathbb S^{d-1})$:
\begin{gather*}
  L^2(\S^{d-1})
  = \bigoplus_{\mathclap{\alpha \in \II \cup \IL }}   \spn \{ e_\alpha \}
  = \bigoplus_{\mathclap{\alpha \in \II \cup \IL }}   \spn \{ \phi_\alpha \}.
  \shortintertext{Setting}
  \HH_\alpha
  = \{
  f(r) e_\alpha (\theta) \mid f \in L^2(\R_+ ; r^{d-1} \dd r)
  \} , \quad \alpha \in \II \cup \IL , \\
  \tilde{\mathcal  {H}}_\alpha
  = \{
  f(r) \phi_\alpha (\theta) \mid f \in L^2(\R_+ ; r^{d-1} \dd r)
  \},  \quad \alpha \in \II \cup \IL ,
\end{gather*}
we get that
$L^2(\R^d)
  = \bigoplus_{\alpha \in \II \cup \IL }   \HH_\alpha
  = \bigoplus_{\alpha \in \II \cup \IL }   \tilde{\HH}_\alpha
$.
In other words, any function $f\in L^2(\mathbb R^d)$ can be written as
\[
  f(x)
  = \sum_{\mathclap{\alpha\in\II\cup\IL}} f_\alpha (r) e_\alpha(\theta)
  = \sum_{\mathclap{\alpha\in\II\cup\IL}} \tilde f_\alpha (r) \phi_\alpha(\theta)
\]
where
$f_\alpha(r) =
  \int_{\S^{d-1}} f(r,\theta) \bar{e}_\alpha(\theta) \dd \theta$,
and
$\tilde f_\alpha(r) =
  \int_{\S^{d-1}} f(r,\theta) \bar{\phi}_\alpha(\theta) \dd \theta$.

\noindent
We now define the following projection operators to each subspace:
\begin{equation*}
  \begin{split}
    \myH_{\alpha}        & \coloneq\LL_{\A ,0}|_{{\HH}_{\alpha}}=-\partial_r^2-\frac{d-1}r\partial_r+\frac{\mu_{\alpha}^2}{r^2},
    \quad \alpha \in \II \cup \IL ,                                                                                                                                      \\
    \tilde \myH_{\alpha} & \coloneq\LL_{\A ,a}|_{\tilde{\HH}_{\alpha}}=-\partial_r^2-\frac{d-1}r\partial_r+\frac{\nu_{\alpha}^2}{r^2}, \quad  \alpha \in \II \cup \IL ,.
  \end{split}
\end{equation*}
It is well-known that
$\mathbb H_{\alpha}$ and $\tilde{\mathbb H}_{\alpha}$ formally have eigenfunctions
\begin{align*}
  E_{\alpha}(r, \lambda)=(r\lambda)^{-\frac{d-2}{2}} J_{\tilde{\mu}_{\alpha}}(\lambda  r);
  \quad \tilde E_{\alpha}(r, \lambda)=( r\lambda)^{-\frac{d-2}{2}} J_{\tilde{\nu}_{\alpha}}(\lambda  r),
\end{align*}
where the $J_{\nu}$ are Bessel functions of the first kind  of order $\nu$, and \[\tilde{\mu}_{\alpha}=\sqrt{\mu_{\alpha}^{2}+\Big(\frac{d-2}{2}\Big)^{2}}, \quad \tilde{\nu}_{\alpha}=\sqrt{\nu_{\alpha}^{2}+\Big(\frac{d-2}{2}\Big)^{2}}.\] Their corresponding eigenvalue is $\lambda^2$, i.e.
\begin{align}\label{equ:lapeigen}
  \myH_{\alpha} E_{\alpha}(r, \lambda)=\lambda^2 E_{\alpha}(r, \lambda), \quad \tilde\myH_{\alpha}\tilde E_{\alpha}(r, \lambda)=\lambda^2 \tilde E_{\alpha}(r, \lambda).
\end{align}
For a radial function $f(r) \in L^2(\R^+,r^{d-1}\dd r)$, we then define its Hankel transform of order $\nu \geq 0$ as
\begin{align}\label{equ:Hanktranf}
  \hankel_{\nu}f(\lambda)=\int_0^\infty (r\lambda)^{-\frac{d-2}{2}} J_{\nu}(r\lambda) f(r) r^{d-1}\dd r.
\end{align}
It is easy to see that
$$
  \hankel_{\nu}\hankel_{\nu}=I,
  \qquad
  \|f\|_{L^2(\R^+,r^{d-1}\dd r)}=\|\hankel_{\nu}f\|_{L^2(\R^+, \lambda^{d-1}\dd \lambda)}.
$$
\noindent
Furthermore, Hankel transforms diagonalize the differential operators $\myH_{\alpha}$ and $\tilde\myH_{\alpha}$:
\begin{align}
  \hankel_{\tilde{\mu}_{\alpha}}(\myH_{\alpha} f)(\lambda) =\lambda^2  \hankel_{\tilde{\mu}_{\alpha}} f(\lambda), \quad     \hankel_{\tilde{\nu}_{\alpha}}(\tilde \myH_{\alpha} f)(\lambda) =\lambda^2  \hankel_{\tilde{\nu}_{\alpha}} f(\lambda).
\end{align}
Therefore, the spectral measures associated to $\mathcal{L}_{\A,a}$ and $\mathcal{L}_{\A,0}$ are given by
\begin{align*}
  \dd E_{\lambda}(\A, 0) f(r, \theta) & = \sum_{\mathclap{\alpha\in\II\cup\IL}}  \hankel_{\tilde{\mu}_{\alpha}}  f_{\alpha}(\lambda) e_{\alpha}(\theta),\quad \text{and}
  \\
  \dd E_{\lambda}(\A, a) f(r, \theta) & = \sum_{\mathclap{\alpha\in\II\cup\IL}}
  \hankel_{\tilde{\nu}_{\alpha}} \tilde f_{\alpha}(\lambda) \phi_{\alpha}(\theta).
\end{align*}

\subsection{Intertwining Operators}\label{sec:inter}

We are now ready to introduce the core of this manuscript.
We start by defining the following families of linear angular operators:
\begin{align}\nonumber
  W^\theta_{\alpha}:\spn\{e_{\alpha}(\theta)\} & \to  \spn\{\phi_{\alpha}(\theta)\} \\\label{equ:Wjthetadef}
  W_{\alpha}^\theta(e_{\alpha}(\theta)) & =
  \phi_{\alpha}(\theta).
\end{align}
The dual operators $\{W_{\alpha}^{\theta,*}\}_{\alpha\in\II\cup  \IL}$ are given by
\begin{align}\nonumber
  W^{\theta,*}_{\alpha}:\spn\{\phi_{\alpha}(\theta)\} & \to \spn\{e_{\alpha}(\theta)\} \\\label{equ:Wjthetastardef}
  W_{\alpha}^{\theta, *}(\phi_{\alpha}(\theta)) & =
  e_{\alpha}(\theta).
\end{align}
\begin{remark}\label{rem:fund}
  Notice that the above operators in are not uniquely defined: indeed, definition \eqref{equ:Wjthetadef} depends on the way we order the eigenfunctions $e_\alpha,\varphi_\alpha$, when the corresponding eigenvalues are repeated (see Assumption \ref{eigenvalue assumption} above). Nevertheless, we are able to prove that any of these choices of the order leads to the definition of intertwining operators with suitable boundedness properties in $L^p$.
\end{remark}
For each fixed $\alpha  \in \II\cup  \IL$, $W_{\alpha}^\theta$ and $W_{\alpha}^{\theta,*}$ are unitary operators on $L^2(\mathbb S^{d-1})$ since both $e_{\alpha}(\theta)$ and $\phi_{\alpha}(\theta)$ are normalized.
\vskip0.2cm
\noindent
Next, we define a family of radial operators $\{W_{\alpha}^r\}_{\alpha \in \II \cup \IL}$ on $L^2(r^{d-1}\dd r)$ as follows:
\begin{equation}\label{equ:wjrdef}
  W_{\alpha}^r\coloneq\hankel_{{\tilde{\nu}_{\alpha}}}\hankel_{\tilde{\mu}_{\alpha}}.
\end{equation}
As for the dual operators $W^{r,*}_\alpha$, those are given by
\begin{equation}\label{equ:wjrastdef}
  W_{\alpha}^{r,*}=\hankel_{\tilde{\mu}_{\alpha}}\hankel_{\tilde{\nu}_{\alpha}}.
\end{equation}
Thanks to the properties of the Hankel transform, we see that for each fixed $j$, $W_{\alpha}^r$ and $W_{\alpha}^{r, *}$ are unitary operators on $L^2(r^{d-1} \dd r)$.
\vskip0.2cm
\noindent
We can finally give the main definition of this manuscript.
\begin{definition}[Intertwining operators]
  Given $f \in L^2(\mathbb R^d)$, with expansions
  \begin{align*}
    f(x)
    =  \sum_{\mathclap{\alpha\in\II\cup\IL}} f_\alpha (r) e_\alpha(\theta)
    = \ \sum_{\mathclap{\alpha\in\II\cup\IL}} \tilde f_\alpha (r) \phi_\alpha(\theta),
  \end{align*}
  we define the intertwining operators $W$, $W^*$ on $ L^2(\mathbb R^d)$ by
  \begin{align}\label{equ:Wdef}
    Wf(x) & \coloneq  \sum_{\mathclap{\alpha\in\II\cup\IL}} W^r_\alpha f_\alpha(r) W^\theta_\alpha(e_\alpha(\theta))
    =  \sum_{\mathclap{\alpha\in\II\cup\IL}} \hankel_{{\tilde{\nu}_{\alpha}}}\hankel_{\tilde{\mu}_{\alpha}}f_\alpha (r) \phi_\alpha(\theta)\,
    \shortintertext{with its dual operator $W^*$ formally given by}
    \label{equ:Wastdef}
    W^*f(x) & \coloneq \sum_{\mathclap{\alpha\in\II\cup\IL}} W^{r,*}_\alpha \tilde f_\alpha(r) W^{\theta,*}_\alpha(\phi_\alpha(\theta))= \sum_{\mathclap{\alpha\in\II\cup\IL}} \hankel_{{\tilde{\mu}_{\alpha}}}\hankel_{\tilde{\nu}_{\alpha}}f_\alpha (r) e_\alpha(\theta).
  \end{align}
\end{definition}
In order to check the duality, notice that,
for any  $f, g \in L^2(\mathbb R^d)$,
one has
\begin{align*}
  \langle Wf, g\rangle & =\sum_{\alpha}\langle\hankel_{\tilde{\nu}_\alpha}\hankel_{\tilde{\mu}_\alpha} f_\alpha, \tilde g_\alpha \rangle_{\tilde \HH_\alpha}=\sum_{\alpha}\langle f_\alpha, \hankel_{\tilde{\mu}_\alpha}\hankel_{\tilde{\nu}_\alpha}\tilde g_\alpha \rangle_{\HH_\alpha}= \langle f, W^*g\rangle.
\end{align*}
As the Hankel transform is unitary, and both $\{e_\alpha\}_{\alpha\in \II \cup \IL}$ and  $\{\phi_\alpha\}_{\alpha\in \II \cup \IL}$ are bases, it is easy to see that $W$ is also unitary, i.e.
\begin{align}\label{inversion operator}
  W W^*= W^*W =I.
\end{align}
We now show that the above operator enjoy the above mentioned intertwining property.
\begin{lemma}[Intertwining property]\label{lem:l2bound}
  The operator $W$ defined in \eqref{equ:Wdef} satisfies the following intertwining property in $L^2(\mathbb R^d)$: for any bounded Borel-measurable function $F$,
  \begin{equation}\label{equ:interwinF12}
    F(\mathcal{L}_{\A,a})=WF(\LL_{\A,0})W^*.
  \end{equation}
\end{lemma}
\begin{proof}
  By \eqref{inversion operator},  it is sufficient to show that  $F(\mathcal{L}_{\A,a}) W=WF(\LL_{\A,0})$. By the definition, we have for the left-hand side
  \begin{align*}
    F(\mathcal{L}_{\A,a})W f(r,\theta)
     & =F\big(\mathcal{L}_{\A,a}\big) \Big(\sum_\alpha\hankel_{{\tilde{\nu}_{\alpha}}}\hankel_{\tilde{\mu}_{\alpha}} f_{\alpha}(r ) \phi_{\alpha}(\theta)\Big)                                      \\
     & =\sum_\alpha F(\tilde \myH_{\alpha}) \hankel_{{\tilde{\nu}_{\alpha}}}\hankel_{\tilde{\mu}_{\alpha}} f_{\alpha}(r ) \phi_{\alpha}(\theta)                                                     \\
     & =\sum_\alpha \hankel_{\tilde{\nu}_{\alpha}}\hankel_{\tilde{\nu}_{\alpha}}F(\myH_{\alpha} )\hankel_{\tilde{\nu}_{\alpha}}\hankel_{\tilde{\mu}_{\alpha}} f_{\alpha}(r ) \phi_{\alpha}(\theta)) \\
     & =\sum_\alpha \hankel_{{\tilde{\nu}_{\alpha}}}\left[ F(\lambda^2) (\hankel_{\tilde{\mu}_{\alpha}}f_{\alpha})(\lambda)\right](r ) \phi_{\alpha}(\theta),
  \end{align*}
  while the right-hand side is
  \begin{align*}
    WF(\LL_{\A,0})f(r,\theta) & =W \bigg(\sum_\alpha  F(\myH_{\alpha}) f_{\alpha}(r )e_{\alpha}(\theta)\bigg)                                                                                    \\
                              & =\sum_\alpha  \hankel_{\tilde{\nu}_{\alpha}}\hankel_{\tilde{\mu}_{\alpha}}\left[ F(\myH_{\alpha}) f_{\alpha }(\cdot)\right](r )\phi_{\alpha}(\theta)             \\
                              & =\sum_\alpha \hankel_{{\tilde{\nu}_{\alpha}}}\left[ F(\lambda^2) \left(\hankel_{\tilde{\mu}_{\alpha}}f_{\alpha}\right)(\lambda)\right](r )\phi_{\alpha}(\theta).
  \end{align*}
  The proof is hence complete.
\end{proof}

\section{Main results}
We showed in the previous section that the intertwining operators defined in \eqref{equ:Wdef} and \eqref{equ:Wastdef} are bounded in $L^2(\mathbb R^d)$.
However, to get the $L^p$-boundedness of those operators for $p\neq 2$, we need stronger spectral conditions than those in Assumption \ref{eigenvalue assumption} on the spherical operators $L_{\A,0}$ and $L_{\A,a}$. Roughly speaking, we need to control the asymptotic size of the separation between the eigenvalues $\mu_{jk}^2, \nu_{jk}^2$. Notice that $\A$ is fixed in the two cases; in fact, we believe that the intertwining operators between $\LL_{\A_1,0}$ and $\LL_{\A_2,0}$ are not bounded in $L^p$ for any nontrivial $\A_1,\A_2$, as  their eigenvalues $\mu_{jk}^2, \nu_{jk}^2$ (appropriately paired up) will not be close enough to each other as $j$ goes to infinity (see the details in the sequel).

\subsection{Discrete mulitiplier bases}
In order to state our main results, we start by some preliminary notations and definitions.
For any doubly indexed sequence $\{\myC_{\alpha}\}_{\alpha \in \IL}$, we define its average sequence $\{\overline{\myC}_j\}_{j\geq \ell}$ by
\begin{align}
    \overline{\myC}_j=\frac{1}{m_j} \sum_{k=1}^{m_j} \myC_{jk}
\end{align}
Then we define the forward finite difference operator with respect to $j$ acting on sequences $\{\overline{\myC}_{j}\}_{j \geq \ell}$ by
\[ \fd  \overline{\myC}_{j}\coloneq  \overline{\myC}_{j+1} - \overline{\myC}_{j}, \quad \fd^N  \overline{\myC}_{j}\coloneq\fd^{N-1}  (\fd \overline{\myC}_{j}) . \]

\begin{definition}\label{defn:dmb} We call a basis   $\{\psi_{\alpha}\}_{\alpha \in \II \cup \IL}$ on $L^2(\mathbb S^{d-1})$ a \emph{discrete multiplier basis} and write $\{\psi_{\alpha}\}_{\alpha \in \II \cup \IL} \in \DMB$ if whenever $\{\myC_{\alpha}\}_{\alpha\in \IL}$ is  a sequence of complex numbers
  satisfying
  \begin{align}
    \sup_{j\geq \ell} \left|\overline{\myC}_{j}\right| & \lesssim 1,  \label{uniform bound}                    \\
    \sup_{L\in\N \colon 2^L \ge \ell } 2^{(N-1)L} \sum_{j=2^L}^{2 ^{L+1}}\left| \fd ^N \overline{\myC}_{j}\right|
     & \lesssim 1,\quad N=1,2, \dots, \left\lfloor\frac{d-1}{2}\right\rfloor \label{uniform difference bound},
  \end{align}
  we have for all $1<p<\infty$ the following inequality,
  \begin{align}\label{equ:constdist}
    \bigg\|\sum_{j=\ell}^\infty  \sum_{k=1}^{m_j}\myC_{jk} F_{jk} \psi_{jk}(\theta)\bigg\|_{L^{p}(\mathbb S^{d-1})} \lesssim \bigg\|\sum_{j=\ell}^\infty   \sum_{k=1}^{m_j} F_{jk} \psi_{jk}(\theta)\bigg\|_{L^{p}(\mathbb S^{d-1})}.
  \end{align}
\end{definition}
\begin{remark} Let $\{\psi_\alpha\}_\alpha$ be the eigenfunctions of $L_{\A,a}$. When ${\A }=0$ and $a(\theta)=0$,  $\{\psi_{\alpha}\}$ are the spherical harmonics  on $\mathbb S^{d-1}$, and $\{\psi_{\alpha}\}_{\alpha}\in \DMB$ by the results in \cite{BC}. More generally, for any $\A  \in C^{\infty}(\mathbb S^{d-1}; \mathbb R^{d})$ and $a \in C^{\infty}(\mathbb S^{d-1}; \mathbb R)$,  $\{\psi_{\alpha}\}_\alpha\in \DMB$ by  \cite[Theorem 5.3.1]{Sogge}.
\end{remark}
We now state a new fundamental tool for our main results.
\begin{lemma}[Variable coefficients discrete multiplier theorem]\label{lem:Dismul}
  Let $\{\myC_{\alpha}(\theta)\}_{\alpha \in \IL}$ be a sequence of bounded functions belonging to  $C^{d-1}(\mathbb S^{d-1})$ and satisfying the following conditions:
  \begin{gather*}
    \sup_{j, \theta}|\overline{\myC}_{j}(\theta)| <C, \quad \sup_{j, \theta}\left|\partial_\theta^{d-1}\overline{\myC}_{j}(\theta)\right| <C,\\
    \sup_{\theta, 2^L \geq \ell} 2^{(N-1)L} \sum_{j=2^L }^{2 ^{L+1}}\left| \fd ^N \overline{ \myC}_{j}(\theta)\right| < C,                        \\
    \sup_{\theta, 2^L \geq \ell} 2^{(N-1)L} \sum_{j=2^L}^{2 ^{L+1}}\left| \fd ^N \partial_\theta^{d-1} \overline{\myC}_{j}(\theta)\right|  <C,
  \end{gather*}
  for $N=1,2, \dots, \left\lfloor\frac{d-1}{2} \right\rfloor$.
  If $\{\psi_{\alpha}\} \in \DMB$, then for every $1<p<\infty$,
  \begin{align*}
    \bigg\|\sum_{j = \ell}^\infty \sum_{k=1}^{m_j} \myC_{jk}(\theta)  F_{jk} \psi_{jk}(\theta)\bigg\|_{L^{p}(\mathbb S^{d-1})} \lesssim \bigg\|\sum_{j = \ell}^\infty   \sum_{k=1}^{m_j} F_{jk} \psi_{jk}(\theta)\bigg\|_{L^{p}(\mathbb S^{d-1})}.
  \end{align*}
\end{lemma}
\begin{proof}
  Let $\theta=\theta(\theta_1,\dots,\theta_{d-1})\in\mathbb S^{d-1}$ be the standard parameterization of the sphere, and abusively write $\myC_{jk}(\theta_1,\dots,\theta_{d-1}):= \myC_{jk}(\theta(\theta_1,\dots,\theta_{d-1}))$.  By the Fundamental Theorem of Calculus, we can write
  \begin{align*}
    \myC_{jk}(\theta)
     & =
    \myC_{jk}(0,\theta_2,\dots,\theta_{d-1})
    +\int_0^{\theta_1} (\partial_\theta \myC_{jk}\partial_{\theta_1} \theta)(\theta_1',\theta_2,\dots,\theta_{d-1})\dd \theta_{1}'
  \end{align*}
  We can repeat this again for the first term $\myC_{jk}(0,\theta_2,\dots,\theta_{d-1})$  to get
  \begin{align*}
    \myC_{jk}(0,0,\theta_3,\dots,\theta_{d-1})
    +\int_0^{\theta_2} \partial_\theta \myC_{jk}(0,\theta_2',\theta_3,\dots,\theta_{d-1})\partial_{\theta_2} \theta(0,\theta_2',\theta_3,\dots,\theta_{d-1})\dd \theta_{2}',
  \end{align*}
  and for the integral term $\int_0^{\theta_1} (\partial_\theta \myC_{jk}\partial_{\theta_1} \theta)(\theta_1',\theta_2,\dots,\theta_{d-1})\dd \theta_{1}'$  to get
  \begin{multline*}
    \int_0^{\theta_1} (\partial_\theta \myC_{jk}\partial_{\theta_1} \theta)(\theta_1',0,\theta_3,\dots,\theta_{d-1})\dd \theta_{1}'  \\
    +
    \int_0^{\theta_2}\int_0^{\theta_1} (\partial_\theta^2 \myC_{jk} \partial_{\theta_1} \theta\partial_{\theta_2}\theta  + \partial_\theta \myC_{jk} \partial_{\theta_1}\partial_{\theta_2}\theta)(\theta_1',\theta_2',\theta_3,\dots,\theta_{d-1})\dd \theta_{1}'\dd\theta_2'.
  \end{multline*}
  We can inductively repeat this until we obtain terms of the form
  \begin{align*}
    c\int_0^{\theta_{i_1}}\dots\int_0^{\theta_{i_N}} G(\theta'_{i_1} ,\dots ,\theta'_{i_N}) \dd\theta'_{i_1} \dots \dd\theta'_{i_N}
  \end{align*}
  where $c$ is a combinatorial constant, $0\le M\le N\le d-1$,  and $G(\theta'_{i_1} ,\dots ,\theta'_{i_N})$ is the product of $\partial^{M}_\theta \myC_{jk}$ and various derivatives of the parameterisation $\theta$, all evaluated at the point
  \[
    (0,\dots,0,\theta'_{i_1},0,\dots,0,\theta_{i_N}',0,\dots,0)\in\mathbb R^{d-1},
  \]
  i.e. the point whose $i_k^\text{th}$ component is equal to $\theta_{i_k}'$, and all other components equal to 0. A precise enumeration could be given with the Fa\`a di Bruno formula, but the important point is that every term depends on the coordinates $\theta_1,\dots,\theta_{d-1}$ only in the integration limits. Thus, every integrand $G(\theta'_{i_1} ,\dots ,\theta'_{i_N})$ can be treated as a constant coefficient multiplier, as they satisfy \eqref{uniform bound} and \eqref{uniform difference bound} for each fixed choice of $\theta'_{i_1} ,\dots ,\theta'_{i_N}$. Hence, by applying Definition \ref{defn:dmb}, we have that for any $1<p<\infty$,
  \begin{align*}
     & \Big  \| \sum_{j,k} \myC_{jk}(\theta)F_{jk} \psi_{jk}(\theta)\Big \|_{L^p}                                                                                                                                           \\
     &
    \begin{multlined}
      \lesssim \Big \|\sum_{j,k}\myC_{jk}(0,\dots,0) F_{jk} \psi_{jk}(\theta)\Big \|_{L^p} \\
      +\!\!\sum_{M,N,i_1,\dots,i_N}\int_0^{\pi}\!\cdots\! \int_0^{\pi}\!\int_0^{2\pi} \Big \| \sum_{j,k}G(\theta'_{i_1} ,\dots ,\theta'_{i_N}) F_{jk} \psi_{jk}(\theta) \Big \|_{L^p}  \dd\theta'_{i_1} \dots \dd\theta'_{i_N}
    \end{multlined} \\
     & \lesssim \bigg\|\sum_{j,k} F_{jk} \psi_{jk}(\theta)\bigg\|_{L^{p}},\end{align*}
  which was what we wanted.
\end{proof}

\subsection{Proper perturbations}\label{subsec:assumLAa}

\begin{definition}[Proper perturbation]\label{def:pp}Assume that $ L_{\A, a}$ and $L_{\A, 0}$ satisfy Assumption \ref{eigenvalue assumption}. We call $ L_{\A, a}$ is a \emph{proper perturbation} of $L_{\A,0}$ if
  \begin{enumerate}
    \item There exists a constant $C_a$  that depends only on $\A$ and $a$ such that for all $j\geq \ell$,
          \[
            \left|\nu_{jk}^2 - \mu_{jk}^2-C_a\right| \leq \frac{C}{j}
          \]
          for all $1 \le k \le m_j$, where $C$ is a  constant independent of $j$ and $k$.
    \item \label{eigenfunction property} The corresponding eigenfunctions $\phi_{jk}$ and $e_{jk}$ $j\geq \ell$ can be written
          \[ \phi_{jk}(\theta)=e_{jk}(\theta)(1+ R_{jk}(\theta)),\]  where $R_{jk}$ satisfies with $N=\big\lfloor\frac{d-1}2\big\rfloor+1$ the following estimates:
          \begin{enumerate}
            \item $\sup_{\theta}|\overline{R}_{j}(\theta)|=O(\frac{1}{j})$ as $j\to \infty$, where  $\overline{R}_j=\frac{1}{m_j} \sum_{k=1}^{m_j} R_{jk}$,
            \item
                  $\sup_{2^L\geq \ell, \theta}2^{(N-1)L}\sum_{j=2^L}^{2^{L+1}}|\mathfrak D^N \overline{R}_{j}(\theta)|  \lesssim \ell^{-1}$,
            \item  $\sup_{2^L \geq \ell, \theta\in \mathbb S^{d-1}}|\partial_\theta^{d-1} \overline{R}_{j}(\theta)|\lesssim \ell^{-1}$,
            \item $\sup_{2^L\geq \ell, \theta}2^{(N-1)L}\sum_{j=2^L}^{2^{L+1}}|\mathfrak D^N \partial_\theta^{d-1} \overline R_{j}(\theta)|  \lesssim \ell^{-1}$.
          \end{enumerate}
  \end{enumerate}
\end{definition}
\begin{remark} In some cases, in order to estimate the derivatives of the reminder terms, we can further decompose $R_{jk}=\sum_{m=1}^d R_{jk}^{m}$ with the $m^{\text{th}}$ reminder term $R_{jk}^m$ satisfying
  \begin{align*}
    \|R_{jk}^{m}\|_{L^\infty} =O\Big(\frac{1}{j^m}\Big).
  \end{align*}
  For $m\leq d-1$, we can control the derivatives using the explicit formula. For $R_{jk}^d$, we can use triangle inequality and the fact that
  $\sum_{jk} \frac{1}{j^d}\lesssim \sum_{j}\frac1{j^2}<\infty$
  to get the boundedness in $L^p$. We will see this in Section \ref{subsec:AinZ}.
\end{remark}

\begin{remark} When $d=2$, $\A \in W^{1, \infty}(\mathbb S^1)$ and $a\in W^{1, \infty}(\mathbb S^1)$, $L_{\A, a}$ is a proper perturbation of $L_{\A, 0}$ by the results in \cite{FFFP}.
  \todo{In higher dimensions, we do not have any nontrivial examples yet.}
\end{remark}
\begin{remark}
  An important question that we will not address in the present paper is the following:
  for a fixed $\A$, what is the minimal regularity of $a(\theta)$ such that $L_{\A, a}$ is a proper perturbation of $L_{\A,0}$?
\end{remark}
\begin{lemma}Assume that $\{e_\alpha\} \in \DMB$. If $ L_{\A, a}$ is a proper perturbation of $ L_{\A, 0}$, then $\{\phi_\alpha\} \in \DMB$.  That is,
  whenever $\{\myC_{jk}\}_{(j,k) \in \IL}$ satisfies the estimates  \eqref{uniform bound} and \eqref{uniform difference bound},
  the perturbed eigenfunctions $\{\phi_{jk}\}$ satisfy
  \begin{align}
    \bigg\|\sum_{j = \ell}^\infty  \sum_{k=1}^{m_j}\myC_{jk} F_{jk} \phi_{jk}(\theta)\bigg\|_{L^{p}(\mathbb S^{d-1})} \lesssim \bigg\|\sum_{j = \ell}^\infty   \sum_{k=1}^{m_j} F_{jk} \phi_{jk}(\theta)\bigg\|_{L^{p}(\mathbb S^{d-1})}.
  \end{align}
\end{lemma}
\begin{proof}
  From \(\phi_{jk}=e_{jk}(\theta)(1 + R_{jk}(\theta))\),  triangle inequality and \(\{e_\alpha\}_\alpha \in \DMB\), we apply Lemma \ref{lem:Dismul} with $\myC_{jk} R_{jk}$ as the variable coefficients to get
  \[
    \bigg\|\sum_{j,k} \myC_{jk} F_{jk} \phi_{jk}\bigg\|_{L^p}
    \lesssim \bigg\|\sum_{j,k} \myC_{jk} F_{jk} e_{jk} \bigg\|_{L^p}
    + \bigg\|\sum_{j,k} \myC_{jk} R_{jk} F_{jk} e_{jk} \bigg\|_{L^p}
    \lesssim \bigg\|\sum_{j,k} F_{jk} e_{jk} \bigg\|_{L^p} .
  \]
  To finish, we write \(e_{jk} = \phi_{jk} - e_{jk} R_{jk}\) and use the fact that \(R_{jk}\) is small for $j\gg1$,
  \begin{align*}
    \bigg\|\sum_{j,k} F_{jk} e_{jk}\bigg\|_{L^{p}} & \leq \bigg\|\sum_{j,k} F_{jk} \phi_{jk}\bigg\|_{L^{p}}+\frac{C}{\ell}\bigg\|\sum_{j,k} F_{jk} e_{jk}\bigg\|_{L^{p}},
  \end{align*}
  so the second term on the right hand side can be absorbed into the left side. So the result is proven.
  %
\end{proof}


\subsection{Main result}
We are now ready to state our main result. From now on, we let $n(d):=\frac{d-2}2$.

\begin{theorem}[$L^p$-boundedness of $W$ and $W^*$]\label{Main result} Let $d\geq2$,  and  $L_{\A, 0}$,  $L_{\A, a}$ satisfy Assumption \ref{eigenvalue assumption}.  Assume that the eigenfunctions of $L_{\A,0}$,  $\{e_\alpha\}_\alpha\in\DMB$ and $ L_{{\A, a}}$ is a proper perturbation of $ L_{{\A, 0}}$. Then
  $W$ defined in \eqref{equ:Wdef} is bounded on $L^p(\R^d)$ for \(p\in(1,\infty)\) satisfying
  \[
    \frac{n(d)-\tilde{\nu}_1}{d}<\frac{1}{p}<\frac{d-n(d)+\tilde{\mu}_1}{d},
  \]
  %
  and $W^*$ defined in \eqref{equ:Wastdef} is bounded on $L^p(\mathbb R^d)$ for \(p\in (1,\infty)\) satisfying
  \[
    \frac{n(d)-\tilde{\mu}_1}{d}<\frac{1}{p}<\frac{d-n(d)+\tilde{\nu}_1}{d},
  \]
  and $\tilde{\mu}_1=\sqrt{\lambda_1(\A ,0)+n(d)^2}$, $\tilde{\nu}_1=\sqrt{\lambda_1(\A ,a)+n(d)^2}$. In particular, if $\tilde{\mu}_1, \tilde{\nu}_1 \geq n(d)$, then $W$ and $W^*$ are bounded on $L^p(\mathbb R^d)$ for all $1<p<\infty$.
\end{theorem}
\begin{remark}
  Note that the restriction on \(p\) depends on the angular operators $L_{\A,0}, L_{\A,a}$ only through their first eigenvalues.\end{remark}
\begin{remark}
  Owing to our use of the theory of singular integrals and the discrete multiplier theorem, we are unable to prove the endpoint cases $p=1,\infty$ with this method.
\end{remark}
\begin{remark}
  For $d\ge3$, Theorem \ref{Main result} includes the inverse square potential which is given by $\A=0$ and $a$ constant. For more general potentials, our result shows that to study $\LL_{\A,a}$ it suffices to understand $\LL_{\A,0}$ and the perturbation theory for the angular component.
\end{remark}

\subsection{The two-dimensional case}

In dimension $d=2$, the spectral assumptions in our main theorem are generically satisfied, which leads to the following quite complete result. Let us define $A,\tilde a$, and $\tilde A$ by\begin{equation}\label{equ:alpha}
  \begin{gathered}
    A(\theta)=\A(\cos\theta,\sin\theta)\cdot (-\sin\theta,\cos\theta),\\
    \tilde a=\frac{1}{2\pi}\int_0^{2\pi} a(\theta)\dd \theta, \qquad
    \tilde A=\frac{1}{2\pi}\int_0^{2\pi} A(\theta)\dd\theta.
  \end{gathered}
\end{equation}
\begin{theorem}   \label{results in dimension two}
  Let $a\in W^{1,\infty}(\mathbb{S}^{1},\mathbb{R})$, $\A \in W^{1,\infty}(\mathbb{S}^{1},\R^2)$ satisfying \eqref{eq:transversal}  and \eqref{equ:a-Aassum} hold. If $\tilde A \notin \frac{1}{2} \mathbb Z$ or $\tilde A \in \frac{1}{2} \mathbb Z$ and $a(\theta)$ is symmetric across $\theta=\pi$ (i.e. $a(\pi-\theta)=a(\pi+\theta)$ for all $\theta\in[0,\pi]$), $W$ and $W^*$ defined in \eqref{equ:Wdef} and \eqref{equ:Wastdef}  are $L^p(\R^2)$-bounded for all $1<p<\infty$.
\end{theorem}
Theorem \ref{results in dimension two} will be proven in Section \ref{sec: dimension two} using Theorem \ref{Main result}. We now state some of the consequences of Theorem \ref{results in dimension two}.


For the wave propagator, the authors in \cite{FZZ} has proved the following estimate for the electric-free operator,
\begin{equation}\label{equ:dispwala0}
  \bigg\|\frac{\sin(t\sqrt{\mathcal{L}_{\A,0}})}{\sqrt{\mathcal{L}_{\A,0}}}\varphi(\sqrt{\mathcal{L}_{\A,0}})f\bigg\|_{L^\infty(\R^2)}
  \leq C(1+|t|)^{-\frac12}\|f\|_{L^1(\R^2)},
\end{equation}
where $\text{supp}\;\varphi\subset[1/2,2].$ Interpolating this with trivial $L^2$ bound, we have  for any $1\leq p\leq2$
\begin{equation*}
  \bigg\lVert\frac{\sin(t\sqrt{\mathcal{L}_{\A,0}})}{\sqrt{\mathcal{L}_{\A,0}}}\varphi\big(\sqrt{\mathcal{L}_{\A,0}}\big)f\bigg\rVert_{L^{p'}(\R^2)}
  \leq C(1+|t|)^{-\frac12(\frac1p-\frac1{p'})}\|f\|_{L^p(\R^2)},
\end{equation*}
where $\frac1p+\frac1{p'}=1$.  This together with Theorem  \ref{results in dimension two} and the intertwining property  \eqref{equ:interwinF12} gives the corresponding result for $\mathcal L_{\A,a}$:
\begin{corollary}[Dispersive estimates of wave propagator]\label{prop:disp}
  Suppose that $\A(\theta)$ and $a(\theta)$ satisfy the assumptions of Theorem \ref{results in dimension two}. Let $\varphi\in C_c^\infty(\R\setminus\{0\})$ with
  $0\leq\varphi\leq1$ and $\text{supp}\;\varphi\subset[1/2,2]$. Then, for $1<p\le 2$,
  \begin{equation}\label{equ:dispwala0Lp}
    \bigg\lVert\frac{\sin(t\sqrt{\mathcal{L}_{\A,a}})}{\sqrt{\mathcal{L}_{\A,a}}}\varphi(\sqrt{\mathcal{L}_{\A,a}})f\bigg\rVert_{L^{p'}(\R^2)}
    \leq C(1+|t|)^{-\frac12(\frac1p-\frac1{p'})}\|f\|_{L^p(\R^2)}.
  \end{equation}
\end{corollary}

Recently, Fanelli--Zhang--Zheng \cite{FZZ23} extended the uniform resolvent estimates for Laplacian  to the pure magnetic potential. They proved that for all \(p\) and  \(q\) satisfying
\begin{equation}\label{equ:pq2d}
  \frac23\le \frac1p-\frac1q <1, \ \frac34<\frac1p\le1, \ \text{and} \ 0\le \frac1q<\frac14,
\end{equation}
the following resolvent estimate is valid,
\begin{equation}\label{equ:ureLA0}
  \left\|(\mathcal{L}_{\A,0}-z)^{-1}f\right\|_{L^{q} (\R^2)}\leq
  C|z|^{\frac1p-\frac1q-1}\|f\|_{L^{p}(\R^2)},\quad \text{ for all }\;z\in\C\setminus \R^+.
\end{equation}
An immediate consequence of  Theorem \ref{results in dimension two} and \eqref{equ:ureLA0} is

\begin{corollary}[Uniform resolvent estimate]\label{prop:ure}
  Assume that $\A(\theta)$ and $a(\theta)$ satisfy the assumptions of Theorem  \ref{results in dimension two}. Let $(p,q)$ satisfy \eqref{equ:pq2d} and $1<p, q<\infty$. Then,  we have
  \begin{equation}\label{equ:ureLAa}
    \left\|(\mathcal{L}_{\A,a}-z)^{-1}f\right\|_{L^{q} (\R^2)}\leq
    C|z|^{\frac1p-\frac1q-1}\|f\|_{L^{p}(\R^2)},\quad \text{ for all }\;z\in\C\setminus \R^+.
  \end{equation}

\end{corollary}


The Bochner--Riesz means of order $\delta$ associated with the any self-adjoint operator $H$ defined using spectral theory by
\begin{equation}\label{equ:BRmesLAa}
  S_R^\delta(H)=\Big(1-\frac{H}{R^2}\Big)^\delta_+,
\end{equation}
with $\delta\geq0,\;R>0$.
For the free Laplacian in dimension two, L. Carleson and P. Sj\"olin \cite{CS72} proved that for $p \in [1,2)\cup (2,\infty]$, the Bochner--Riesz means $S_R^\delta(-\Delta)f$ converges to $f$ in  $L^p(\R^2)$ if and only if
\begin{equation}\label{equ:panddel}
  \delta>\delta_c(p,2)\coloneq\max\Big\{0,2\Big|\frac12-\frac1p\Big|-\frac12\Big\},
\end{equation}
and $\delta\geq0$ when $p=2$. This result follows from the boundedness of $S^{\delta}_1(\Delta)$,
\begin{equation}\label{equ:BRRnbdLp}
  \|S^\delta_1(-\Delta)f\|_{L^p(\R^2)}\leq C\|f\|_{L^p(\R^2)},\quad \delta>\delta_c(p,2).
\end{equation}

Recently, Miao--Yan--Zhang \cite{MYZ} extended this to the pure magnetic case, that is
\begin{equation}\label{equ:BRRnbdLA0}
  \|S^\delta_1(\mathcal{L}_{\A,0})f\|_{L^p(\R^2)}\leq C\|f\|_{L^p(\R^2)},\quad \text{ for all }\;1\leq p\leq\infty,\;\delta>\delta(p,2).
\end{equation}
\begin{corollary}[$L^p$ boundedness for Bochner--Riesz means]\label{prop:BRLAa}
  Suppose that $\A(\theta)$ and $a(\theta)$ satisfy the assumptions of Theorem  \ref{results in dimension two}. Let $\delta>\delta_c(p,2)$. Then, we have
  \begin{equation}\label{equ:BRRnbdLAa}
    \|S^\delta_1(\mathcal{L}_{\A,a})f\|_{L^p(\R^2)}\leq C\|f\|_{L^p(\R^2)},\quad \text{ for all }\;1<p<\infty.
  \end{equation}
  As a consequence, we deduce that $S_R^\delta(\mathcal{L}_{\A,a})f$ converges to $f$ in  $L^p(\R^2)$ as $\delta>\delta_c(p,2)$ and $1<p<\infty.$

\end{corollary}

%
%
%


\section{Preliminary tools}\label{Sec:Prelim}
In this section we give some preliminaries for the proof of Theorem \ref{Main result}.
Let $\Gamma(z)=\int_0^\infty e^{-t} t^{z-1} \dd t$ be the Gamma function.

\begin{lemma}\label{lem:behab}
  Let $a>0$ and $|b|\leq1$. Then,
  \begin{equation}\label{equ:behab}
    \left| \frac{\Gamma(a+n+1)\Gamma(b+n+1)}{\Gamma(a+b+n+1) \Gamma(n+1)}\right|\leq C a^{|b|},\;\quad \text{ for all }\;n\in\N,
  \end{equation}
  as $a\to\infty$, where the constant $C$ is independent of $a$ and $n$.

\end{lemma}

\begin{proof}
  It follows (see \cite{Tricomi}) from Stirling's formula that
  \(
  \big| \frac{\Gamma(a+n+1)}{\Gamma(a+b+n+1)}\big|  \leq C(a+n+1)^{-b}
  \)
  and
  \(
  \big|\frac{\Gamma(b+n+1)}{\Gamma(n+1)}\big|       \leq C(n+1)^b
  \), which implies the result.
\end{proof}

\subsection{Mellin multipliers} Recall the Hankel transform \(\hankel_\nu\) of order \(\nu\)  from \eqref{equ:Hanktranf}.
Given the definitions \eqref{equ:wjrdef}--\eqref{equ:Wastdef}, we are led to consider the composition $T_{\nu,\mu}  \coloneq \hankel_{\nu } \hankel_{\mu} $.
To study the $L^p$-boundedness of this operator,  we will use the Mellin transform, defined by
\[\mathcal M (f(s))(z)=\int_0^{\infty} f(s)s^{z-1}\dd s,\]
whenever the integral is absolutely convergent.

The Mellin transform turns the Hankel transform into multiplication by a certain ratio of Gamma functions. It follows that the operator $T_{\nu, \mu }$ is the Mellin multiplier
\begin{align*}
  \mathcal M (T_{\nu, \mu} f)(z)=m_{\nu, \mu}(z)\mathcal M ( f)(z)
\end{align*}
where \(n(d)=\frac{d-2}{2}\) as in Theorem \ref{Main result} and
\begin{align}
  m_{\nu, \mu}(z )=\frac{\Gamma(\frac{z+\nu-n(d)}{2}) \Gamma(1-\frac{z-\mu-n(d)}{2})}{\Gamma(1-\frac{z-\nu-n(d) }{2}) \Gamma(\frac{z+\mu-n(d)}{2})}.
\end{align}
This leads to the following $L^p$ boundedness result:
\begin{prop}\label{boundedness of Mellin multiplier}
  Let $\nu, \mu\geq 0$. The operator $T_{\nu, \mu}$ is bounded on $L^p(\R^+,r^{d-1}\dd r)$ for all $p\in(1,\infty)$ such that ${n(d)-\nu}<\frac{d}{p}<{2+n(d)+\mu}$.    In particular, if $\mu, \nu\geq \frac{d-2}{2}$, then $T_{\nu, \mu }$ is bounded on $L^p(\R^+,r^{d-1}\dd r)$ for all $1<p<\infty$.
\end{prop}
\begin{proof}
  The result follows by taking\footnote{This $\mu$ is not the same $\mu$ as in our proposition.} $\mu =-\frac{d}{p}$, $\alpha=n(d)-\nu$, and $\beta=2+n(d)+\mu$ in  \cite[Theorem 2.6]{MS} once we show:
  \begin{enumerate}[(i)]
    \item $m_{\nu, \mu}(z)$ is analytic in $\alpha<\Re z<\beta$,\label{Mellin multiplier condition-1}
    \item $m_{\nu, \mu}(z)$ is bounded in $\alpha<\Re z<\beta$, and \label{Mellin multiplier condition-2}
    \item for $\alpha<\Re z<\beta$,  $|m'_{\nu, \mu}(z)| \leq C_{\nu, \mu} |\Im z|^{-1}$ as $|\Im z|\to \infty$. \label{Mellin multiplier condition-3}
  \end{enumerate}
  \ref{Mellin multiplier condition-2} follows easily from \eqref{equ:behab}, and  \ref{Mellin multiplier condition-1} follows since $\Gamma(z)$ is analytic for $\Re z>0$, $1/\Gamma(z)$ is entire, and  $\Re(1-\frac{z-\mu-n(d)}{2})> 0$ and $\Re \Gamma(\frac{z+\nu-n(d)}{2}) >0$ when $n(d)-\nu<\Re z<2+\mu+n(d)$.

  As for \ref{Mellin multiplier condition-3}, by the formula
  $ f'(z) =(\ln f(z))'f(z)$, one can get
  \begin{align*}
    m_{\nu, \mu }'(z)=\frac{1}{2} m_{\nu, \mu}(z)n_{\nu, \mu}(z) ,
  \end{align*}
  where
  \begin{equation*}
    \begin{multlined}
      n_{\nu,\mu}(z)  =\psi\Big(\frac{z+\nu-n(d) }{2}\Big) -\psi\Big(1-\frac{z-\mu-n(d)}{2}\Big)-\psi\Big(\frac{z+\mu-n(d) }{2}\Big) \\
      \phantom{a} +\psi\big(1-\frac{z-\nu-n(d)}{2}\big),
    \end{multlined}
  \end{equation*}
  and $\psi=\Gamma'/\Gamma$  is the digamma function, which has the asymptotic expansion
  \begin{align*}
    \psi(z)=\ln z-\frac{1}{2z}+O\left(\frac{1}{z^2}\right), \quad \text{as}\quad |z| \to \infty,
  \end{align*}
  This gives the following asymptotic for \(n_{\nu,\mu}\):
  \begin{align*}
    n_{\nu, \mu}(z)
     & =\ln\left(1+\frac{(\mu-\nu)(\mu+\nu+2)}{(z-n(d)+\mu)(z-n(d)-\mu-2)}\right)           \\
     & \quad +
    \frac{2\nu+2}{(z-n(d)+\nu )(z-n(d)-\nu -2)}-\frac{2\mu+2}{(z-n(d)+\mu )(z-n(d)-\mu -2)} \\
     & \quad +O\left(\frac{1}{|z|^2}\right).
  \end{align*}
  Thus,
  \(|n_{\nu, \mu}(z)|\leq C_{\nu, \mu}/|z| \) as \( |z| \to \infty.
  \)
  Using this with the boundedness of $m_{\nu, \mu}(z)$, we see that \ref{Mellin multiplier condition-3} is satisfied, and the proof is completed.
\end{proof}

\subsection{Singular integral operators on the half-line}
%

Following the proof of \cite[Lemma 2.11]{MSZ}, we obtain

\begin{lemma}\label{lem:CZbd} Let $\mu\geq0$ and $\nu\geq0$. For $r\in \mathbb R^{+}$, define
  \begin{align*}
    Tf(r) & =\int_\frac{r}{2}^r  \left(\frac{s}{r}\right)^{\mu+1+\frac{d}{2}-\frac{d}{p}}\frac{ f(s)}{1-(\frac{s}{r})^2} \frac{\dd s}{s}
    -\int_r^{2r}  \left(\frac{r}{s}\right)^{\frac{d}{p}-\frac{d}{2}+1+\nu}\frac{f(s)}{1-(\frac{r}{s})^2}  \frac{\dd s}{s},
  \end{align*}
  then there exists a constant $C=C(\mu,\nu)>0$ such that
  \begin{equation}\label{equ:TbdLp}
    \|Tf\|_{L^p(\mathbb R^{+}, \frac{\dd r}{r})}\leq C\|f\|_{L^p(\mathbb R^{+}, \frac{\dd r}{r})},\quad \text{ for all }\;1<p<\infty.
  \end{equation}

\end{lemma}

\section{Proof of Main Theorem}
To prove Theorem \ref{Main result}, note that it suffices to show that for any $p\in(1,\infty)$ satisfying $\frac{n(d)-\tilde{\nu}_1}{d}\leq \frac{1}{p}\leq \frac{1}{2}$, one has
\begin{equation}\label{equ:redu1wpgeq2}
  \|Wf\|_{L^{p}(\R^d)}\leq C\|f\|_{L^p(\R^d)}.
\end{equation}
This is because $W$ and $W^*$ can be treated in the same way: if we exchange the roles of $\mu_{\alpha}$  with $\nu_{\alpha}$, and $e_{\alpha}$ with $\phi_{\alpha}$, an entirely analogous argument applied to $W^*$ shows that $W^*$ is also $L^p(\mathbb R^d)$-bounded for any $p\in(1,\infty)$ satisfying $\frac{n(d)-\tilde{\mu}_1}{d}\leq \frac{1}{p}\leq \frac{1}{2}$. Then Theorem \ref{Main result} follows immediately by duality. So in the remainder of the paper, we will only consider \(p\) in this range.

Since $L_{\A,a}$ is a proper perturbation of $L_{\A,0}$, when $j\geq \ell$,
\begin{gather}
  |\tilde{\mu}_{jk}-\tilde{\nu}_{jk}|=\frac{|\tilde{\mu}_{jk}^2-\tilde{\nu}_{jk}^2|}{\tilde{\mu}_{jk}+\tilde{\nu}_{jk}} \leq \frac{C}{j } \quad \text{and}\quad
  \phi_{jk}(\theta)=e_{jk}(\theta)+e_{jk}(\theta) R_{jk}(\theta),\label{equ:phiejk1}
  \shortintertext{with}
  \sup_{1\leq k\leq m_j} \| R_{jk}(\theta)\|_{L^\infty(\mathbb S^{d-1})} <\frac{C}{j },\nonumber
\end{gather}
where $\tilde{\mu}_{\smash{jk}}^{2}=\sqrt{\mu_{jk}^{2}+n(d)^2}$, $\tilde{\nu}_{\smash{jk}}^{2}=\sqrt{\nu_{jk}^{2}+n(d)^2}$.\

Without loss of generality, assume that $\ell$ is so large that when $j \geq \ell$, we have
\begin{equation}\label{equ:munudef}
  \tilde{\mu}_{jk}, \tilde{\nu}_{jk} \geq \frac{d-2}{2}+C j, \ \text{for all} \ 1\leq k\leq m_j,
\end{equation} where $C$ is a small constant.

Using \eqref{equ:phiejk1}, we split $\|Wf\|_{L^p(\R^d)}$ into three terms:
\begin{align}\nonumber
  \|Wf\|_{L^{p}(\R^{d})} & =\bigg\lVert\sum_{\alpha\in\II\cup\IL} W_{\alpha}^r (f_{\alpha}(\cdot)\big)(r)W_{\alpha}^\theta\big(e_{\alpha}(\theta))\bigg\rVert_{L^p} \\     
   & \lesssim \sum_{i=1}^{\ell_0}\left\| \hankel_{\tilde{\nu}_i}\hankel_{\tilde{\mu}_i}f_{i}(r)\phi_{i}(\theta)\right\|_{L^p}
  +\bigg\lVert\sum_{j,k} \hankel_{\tilde{\nu}_{jk}}\hankel_{\tilde{\mu}_{jk}}f_{jk}(r) e_{jk}(\theta) \bigg\rVert_{L^p}
  \\
   & \quad +\bigg\lVert\sum_{j,k} \hankel_{\tilde{\nu}_{jk}}\hankel_{\tilde{\mu}_{jk}}f_{jk}(r)R_{jk}(\theta) e_{jk}(\theta) \bigg\rVert_{L^p}\eqqcolon I_1 + I_2 + I_3,
  \label{defn:I1+I2+I3}
\end{align}
where  $\tilde{\mu}_i=\sqrt{\mu_i^2+n(d)^{2}}$ and  $\tilde{\nu}_i=\sqrt{\nu_i^2+n(d)^{2}}$.
\

We will next show that $I_1$, $I_2$, and $I_3$ are each bounded by $\|f\|_{L^p}$.
\subsubsection*{The estimate of \(I_1\)}  Let $\max\{0, \frac{n(d)-\tilde{\nu}_1}{d}\}\leq \frac{1}{p}\leq \frac{1}{2} $.  For each fixed $\alpha\in \II$, it follows from \cite[Lemma 2.1]{FGK15} that
\begin{equation}\label{equ:ejklinfest}
  \|\phi_{\alpha}(\theta)\|_{L^\infty(\mathbb{S}^{d-1})}\lesssim  |\nu_\alpha|^d.
\end{equation}
This implies that
\begin{align}\label{equ:jsmtermell-1}
  \left\| W_{\alpha}^rf_{\alpha}(r)\phi_{\alpha}(\theta)\right\|_{L^p(\R^d)} & \lesssim_\alpha \left\|W_{\alpha}^rf_{\alpha}(r)\right\|_{L^p(r^{d-1}\dd r)} \\\nonumber
   & \lesssim_\alpha \|f_{\alpha}(r)\|_{L^p(r^{d-1}\dd r)},
\end{align}
where the last inequality follows from Proposition \ref{boundedness of Mellin multiplier}, since $\tilde\nu_{\alpha} \geq \tilde\nu_1 \geq 0$.
By  H\"older's inequality,
we get for $p\geq2$
\begin{align}\nonumber
  \|f_{\alpha}(r)\|_{L^{p}(r^{d-1}\dd r)}
   & \lesssim
  \bigg\|
  \bigg(
  \sum_{\alpha}|f_{\alpha}(r)|^{2}
  \bigg)^{1/2}
  \bigg\|_{L^{p}(r^{d-1}\dd r)} \\
   & =
  \bigg\|
  \sum_{\alpha}f_{\alpha}(r)e_{\alpha}(\theta)
  \bigg\rVert_{L^p(r^{d-1} \dd r ; L^{2}(d\theta))}
  \notag                        \\
   & \lesssim
  \bigg\|
  \sum_{\alpha}f_{\alpha}(r)e_{\alpha}(\theta)
  \bigg\|_{L^{p}(r^{d-1}\dd rd\theta)}
  =\|f\|_{L^{p}(\R^{d})}.
  \label{subcoefficient-1}
\end{align}
Hence,
\begin{equation}\label{equ:Wf1estok}
  I_1\lesssim \sum_{i=1}^{\ell _0}C_i \|f_i\|_{L^{p}(\R^{d})}\lesssim_\ell \|f\|_{L^{p}(\R^{d})}.
\end{equation}
\subsubsection*{The estimate of \(I_3\)}  Since $L_{\A,a}$ is a proper perturbation of $L_{\A,0}$, by the assumption that $\{e_{\alpha}\}_{\alpha}\in\DMB$, we immediately get that $I_3\lesssim  I_2.$

\subsubsection*{The estimate of \(I_2\)} This is given by the following proposition.

\begin{proposition}\label{prop:keypropMSZ}
  Assume that the eigenbasis $\{e_\alpha\}_{\alpha\in \II\cup \IL} $ of $L_{\A ,0}$  belongs to  $\DMB$. Then for all $1<p<\infty$,
  \begin{align}\label{equ:MSZestkey}
    \!\!\!\!\!\!\!\! I_2 = \bigg\lVert\sum_{j = \ell}^{\infty}\sum_{k=1}^{m_j} \hankel_{\tilde{\nu}_{jk}}\hankel_{\tilde{\mu}_{jk}}f_{jk}(r) e_{jk}(\theta) \bigg\rVert_{L^p(\R^{d})}\lesssim\bigg\lVert\sum_{j =\ell}^\infty\sum_{k=1}^{m_j} f_{jk}(r)e_{jk}(\theta)\bigg\rVert_{L^p(\R^{d})}.
  \end{align}
\end{proposition}
\begin{proof}
  This is similar to the proof for the main theorem in \cite{MSZ} so we have put it in Appendix \ref{appendix-proof} for the interested reader.
\end{proof}
Using Proposition \ref{prop:keypropMSZ} and \eqref{subcoefficient-1}, we obtain \begin{align*}
  I_2 & \lesssim\bigg\lVert\sum_{j=\ell}^\infty \sum_{k=1}^{m_j} f_{jk}e_{jk}\bigg\rVert_{L^p}           \lesssim\bigg\lVert\sum_{\alpha\in \II \cup \IL}f_{\alpha}e_{\alpha}\bigg\rVert_{L^p}
  + \sum_{\alpha \in \II}\left\|f_{\alpha}e_{\alpha}\right\|_{L^p}
  \lesssim_\ell\|f\|_{L^p}.
\end{align*}
This completes the proof of Theorem \ref{Main result}. \hfill\(\square\)

\section{The proof of Theorem \ref{results in dimension two}}\label{sec: dimension two}
Assume that $a\in W^{1,\infty}(\mathbb{S}^{1},\mathbb{R})$, $\A \in W^{1,\infty}(\mathbb{S}^{1},\R^2)$ through this section. We will prove Theorem \ref{results in dimension two} using Theorem \ref{Main result}. To do this, we need to verify Assumption \ref{eigenvalue assumption}, show that $\{e_\alpha\}_{\alpha\in \II\cup\IL} \in \mathcal{DMB}$, and prove that $ L_{\A, a}$ is a proper perturbation of $ L_{\A, 0}$. We split the proof into the following cases:
\begin{enumerate}
  \item $\tilde A\notin \frac12\mathbb Z$,
  \item $\tilde A\in\mathbb Z$ and $a$ is symmetric across $\pi$, or

  \item $\tilde{A}\in\frac12\Z\setminus\Z$  and $a$ is symmetric across $\pi$.
\end{enumerate}
In the proof, we will need properties of the $T$-periodic eigenvalue problem.
\begin{definition}
  For each $a \in W^{1,\infty}(\mathbb S^1, \mathbb R)$ and $T\in(-1/2,1/2]$, the \emph{$T$-periodic eigenvalue problem} is the following ODE,
  \begin{equation}
    \begin{cases}
      -\psi''(\theta)+a(\theta)\psi(\theta)
      =\lambda \psi(\theta),\quad \theta\in[0,2\pi] \\
      \psi(2\pi)=e^{2i\pi T}\psi(0),
      \quad \psi'(2\pi)= e^{2\pi i  T}\psi'(0),
    \end{cases}\label{equ:t-periodic-ode}
  \end{equation}
  It admits an increasing sequence $\{\lambda_n(T,a)\}_{n\ge1}$ of eigenvalues $\lambda_n(T,a)\uparrow \infty$, and the following comparison principle from {\cite[Theorem 2.2.2]{Eastham}}:
\end{definition}
\begin{lemma}
  Let $a,b\in W^{1,\infty}$ and $T\in\mathbb R$. If $a \ge b$ pointwise, then for every $n\ge 1$, \[
    \lambda_n(T,a) \ge \lambda_n(T,b).
  \]

\end{lemma}

\subsection{{When  $\tilde A \notin \frac{1}{2}\mathbb Z$}}
Let $\bar A$ be the number in $(-1/2,1/2)\setminus \{0\}$ such that $\bar A - \tilde A \in \mathbb Z$.
We will only consider $\bar A \in (0,\frac12)$, as the case $\bar A<0$ can be treated analogously after adjusting notation (for example, we would need to exchange $\mu_{j1}^2$ and $\mu_{j2}^2$ below to match the order of eigenvalues in Assumption \ref{eigenvalue assumption}\ref{assump1}).
From \cite[Section 2.3]{Eastham}, we know that the free spherical operator $L_{\A ,0}=(-i\partial_\theta+A(\theta))^2$ with periodic boundary condition has eigenvalues
\begin{equation*}
  \mu_0^2=\bar A^2, \quad \mu_{jk}^2
  = (j +(-1)^{k}\bar A)^2,\  \text{for all} \ \ j\geq 1,  k=1, 2.
\end{equation*}
The corresponding cluster sizes are $m_0=1$ and $m_j=2$ for $j\geq 1$,
and the corresponding eigenfunctions are
\begin{align*}
  e_0(\theta)=\frac{1}{\sqrt{2\pi}}e^{i(\bar A\theta - \int_0^\theta A(\theta') \dd \theta')},\qquad
  e_{jk}(\theta)=\frac1{\sqrt{2\pi}}e^{ij (-1)^{k}\theta}e^{i(\bar A\theta - \int_0^\theta A(\theta') \dd \theta')}.
\end{align*}
Clearly, $\{e_{jk}(\theta)\} \in  \mathcal {DMB}$   since  $\{e_0(\theta), \frac1{\sqrt{2\pi}}e^{ij \theta}, \frac1{\sqrt{2\pi}}e^{-ij \theta}\}_{j \geq 1}\in \mathcal {DMB}$.

Let $\{\lambda_n(\A,a)\}_{n\ge1}$ be the eigenvalues of $L_{\A, a}$ arranged as an increasing sequence diverging to infinity. From the proof of \cite[Lemma 2.1]{FFFParxiv}, we know that  there exist $n^*,\ell\in\N$ such that $\lambda_n(\A,a)$ can be written as
\[\{\lambda_n(\mathbf{ A}, a):\;n\in\N,\;n>n^*\}=\{\nu_{jk}^2:\;j\geq \ell, k=1,2\},\] with
\begin{equation}\label{equ:lamjbeha}
  \nu_{jk}^2=\tilde a+\mu_{jk}^2+O\Big(\frac{1}{j^2}\Big),\  \text{as} \ j\to \infty.
\end{equation}
By taking $n^*$ and $\ell$ sufficiently large, we may assume for later use that
\begin{equation}\label{equ:lamjbeha2}
  \lvert \nu_{jk}^2-\tilde a-\mu_{jk}^2\rvert < 1, \text{ for all } j\ge \ell.\end{equation}
It follows from \eqref{equ:lamjbeha} that $L_{\A,0}$ and $L_{\A, a}$ satisfy Assumption
\ref{eigenvalue assumption} \eqref{assump1} with any choice of $C_{\A} \in (0, \frac12 - \bar A)$.

Next, we will show that $L_{\A,0}$ and $L_{\A, a}$ satisfy Assumption
\ref{eigenvalue assumption} \eqref{assump2}.
Starting from the eigenvalue equation,
\begin{equation}\label{equ:LAaeigp}
  \begin{cases}
    L_{\A,a}\phi_n(\theta)=\lambda_n(\A,a)\phi_n(\theta),\quad \theta\in[0,2\pi] \\
    \phi_n(0)=\phi_n(2\pi),\;\phi_n'(0)=\phi_n'(2\pi),
  \end{cases}
\end{equation}
we can use the gauge transformation
\begin{equation}\label{equ:gaugepsi}
  \psi(\theta) = e^{-i\int_0^\theta A(\theta')\dd \theta'}\phi(\theta),
\end{equation}
to obtain the $\bar{A}$-periodic eigenvalue problem \eqref{equ:t-periodic-ode}, with $\lambda_n(\A,a) = \lambda_n(\bar A,a)$.

For $j$ sufficiently large, the eigenvalues $\lambda_{2j}(\bar{A}, a)$ and $\lambda_{2j+1}(\bar{A}, a)$ are equal to $\nu^2_{j'k'}$ for some $j',k'$, and the goal now is to show that we can identify $j'$ and $k'$ for $\lambda_{2\ell }(\bar{A}, a)$ and $\lambda_{2\ell+1}(\bar{A}, a)$.
Notice that $(\spn_{j\geq \ell}\{e_{jk}\})^\perp$ is spanned by the $2\ell-1$ vectors $\{e_0\}\cup\{e_{jk}:\substack{j=1,2,\dots,\ell,\\ k=1,2}\}$. So we need to show that $n^*=2\ell-1$.

Let $a_*=-\|a\|_{L^\infty}-1$ and $a^*=\|a\|_{L^\infty}+1$.
By the comparison principle, we have that for all $n\geq 0$,
\begin{equation}\label{equ:prilamt}
  \lambda_{n}(\bar{A}, a_*) <\lambda_{n}(\bar{A}, a)<\lambda_{n}(\bar{A}, a^*).
\end{equation}
Since $\lambda_{2j}(\bar A,0) = \mu_{j1}^2$ and $\lambda_{2j+1}(\bar A,0)=\mu_{j2}^2$,
we have
\begin{gather*}
  \mu_{j1}^2+a_*<\lambda_{2j}(\bar{A}, a)<\mu_{j1}^2+a^*, \mathrlap{\text{ and}} \\
  \mu_{j2}^2+a_*<\lambda_{2j+1}(\bar{A}, a)<\mu_{j2}^2+a^*.
\end{gather*}

Define for $j\geq\ell$ the intervals $I_j=[\mu_{j1}^2+a_*, \mu_{j2}^2+a^*]$, which contain $\lambda_{2j}(\bar A,a)$ and $\lambda_{2j+1}(\bar A,a)$. By taking $\ell\gg1$ if necessary, the $I_j$ are pairwise disjoint. Note that by \eqref{equ:lamjbeha2}, $I_j$ contains precisely two eigenvalues
\[
  \{\nu^2_{j'k'}: j'\ge \ell,\ k'=1,2 \}\cap I_j = \{  \nu_{j1}^2,\ \nu_{j 2}^2\}.
\]
Hence, $\nu_{j1}^2=\lambda_{2j}$ and $\nu_{j2}^2=\lambda_{2j+1}$ for $j\geq \ell$. In particular, for $j=\ell$, we have $n^*=2\ell -1$.
This proves that $L_{\A,0}$ and $L_{\A, a}$ satisfy Assumption
\ref{eigenvalue assumption} \eqref{assump2} with $\ell_0=2\ell-1$.

Now we check the proper perturbation property. By \cite[Lemma 2.1]{FFFParxiv}, the corresponding normalized eigenfunctions $\phi_{jk}$  can be written as \begin{equation}\label{equ:phijbeh}
  \phi_{jk}(\theta)=e_{jk}(\theta)\big(1+R_{jk}(\theta)\big)
\end{equation}
with
\begin{equation}\label{equ:Cjthetadef}
  R_{jk}(\theta) = e^{(-1)^{k} i( \sqrt{\nu_{\smash{jk}}^2 - \tilde a} - \mu_{jk} )\theta} e^{i \int_0^\theta W_{jk}(\theta') \dd \theta' } - 1
\end{equation}
where $W_{jk}$ satisfies
\begin{equation}\label{equ:Wlambdajthe}
  \big\|W_{jk}\big\|_{C^0(\mathbb{S}^{1})}\leq\frac{C}{\sqrt{\nu_{\smash{jk}}^2-\tilde{a}}}\leq \frac{C}{j },\quad\text{as}\quad j \to\infty.
\end{equation}
A simple calculation gives

\[
  R'_{jk} = (-1)^{k}i(R_{jk}(\theta) + 1 ) \big(\sqrt{\nu_{\smash{jk}}^2 - \tilde a} - \mu_{jk} + W_{jk}(\theta)\big).
\]
By \eqref{equ:lamjbeha} and \eqref{equ:Wlambdajthe}, we get for any $\theta\in[0,2\pi]$ and $k=1, 2$
\begin{equation}\label{equ:cjlargest}
  |R_{jk}(\theta)|\leq \frac{C}{j },\quad  |R_{jk}'(\theta)|\leq \frac{C}{j },\quad\text{as}\quad j \to\infty.
\end{equation}
Hence,
\begin{gather}\label{equ:cjthetBCcond}
  \sup_{\substack{j \geq\ell\\ \theta\in [0,2\pi]}}
  \lvert R_{j1}(\theta)\rvert
  +\lvert R_{j2}(\theta)\rvert
  +\lvert R_{j1}'(\theta)\rvert
  +\lvert R_{j2}'(\theta)\rvert
  \leq \frac{C}{\ell}, \\
  \shortintertext{and}
  \sup_{\substack{j \geq\ell\\ \theta\in [0,2\pi]}}
  \sum_{j =2^L}^{2^{L+1}}
  \Big(
  \lvert \fd R_{j1}'(\theta)\rvert
  +\lvert \fd R_{j2}'(\theta)\rvert
  \Big)\leq \frac
  {C}{\ell}.
\end{gather}
In conclusion,  $L_{\A, a}$ is a proper perturbation of $L_{\A, 0}$.

The conditions for Theorem \ref{Main result} have been verified, and we deduce that  the wave operator $W$ and $W^*$ are bounded on $L^p(\mathbb R^2)$ for $1<p<\infty$ due to $\mu_1,\nu_1\geq0.$ This proves Theorem \ref{results in dimension two} when $\tilde{A}\not\in\frac12\Z.$

\subsection{{When $\tilde{A}\in\Z$ and $a$ is symmetric across $\pi$}}\label{subsec:AinZ} In this case, the eigenvalues and eigenfunctions of $L_{\A,0}$ are given by
\begin{align*}
  \label{equ:defmujkejk}
  \mu_0^2    & =0,    & e_0(\theta)    & =\frac{1}{\sqrt{2\pi}},                                        \\
  \mu_{j1}^2 & =j^2,  & e_{j1}(\theta) & =\frac{1}{\sqrt{\pi}}e^{-i\int_0^{\theta}A(\theta')\dd\theta'}
  \sin j\theta,                                                                                         \\
  \mu_{j2}^2 & = j^2, & e_{j2}(\theta) & =\frac{1}{\sqrt{\pi}}e^{-i\int_0^{\theta}A(\theta')\dd\theta'}
  \cos j\theta.
\end{align*}
For the perturbed operator $L_{\A,a}\varphi=(-i\partial_\theta+A(\theta))^2\varphi+a(\theta)\varphi$,  by \cite[ Lemma B.7]{FFFParxiv}, there exists $\ell\in\N$ such that $\{\lambda_n:\; n \in\N,\;n \geq n^*\}=\{\nu_{jk}^2: j \geq\ell,\;k=1,2\}$ and
\begin{equation}\label{equ:lambdjdefA}
  \nu_{jk}^2=\mu_{jk}^2+\tilde{a}+O\Big(\frac1{j }\Big),\quad\text{as}\quad j \to\infty.
\end{equation}
Using the comparison principle as before, one can prove that $n^*=2\ell-1$, so  $L_{\A, 0}$ and
$L_{\A, a}$ satisfy the Assumption \ref{eigenvalue assumption}.

Under the symmetry assumption, $a$ is given by a cosine series.  Let $a_{\myc, k}$ denote the $k^\text{th}$ cosine-Fourier coefficient of $a(x)$ and write
\begin{align}\label{equ:Cj1theta}
  R_{jk,\mys}(\theta)= & \frac{1}{2}\sum_{m\neq j}\frac{a_{\myc,|j-m|}+(-1)^k a_{\myc,j+m}}{(j-m)(j+m)}\sin((m-j)\theta), \\
  R_{jk,\myc}(\theta)= & \frac{1}{2}\sum_{m\neq j}\frac{a_{\myc,|j-m|}+(-1)^k a_{c,j+m}}{(j-m)(j+m)}\cos((m-j)\theta).
\end{align}
Direct computation gives
\begin{equation}\label{equ:Cj120}
  \sup_{\theta}\{|R_{jk,\mys}(\theta)|, |R_{jk,\myc}(\theta)|, |R_{jk,\mys}'(\theta)|, |R_{jk,\myc}'(\theta)|\}\lesssim \frac{1}{j}.
\end{equation}
Then the corresponding eigenfunctions $\phi_{jk}$
can be written as
\begin{align}\label{equ:varphiejrjA}
  \phi_{j1}(\theta)=e_{j1}(\theta)\big(1+R_{j1, \myc}(\theta)+R_{j1}^2(\theta)\big)+ e_{j2}(\theta) R_{j1, \mys}(\theta), \\
  \phi_{j2}(\theta)=e_{j2}(\theta)\big(1+R_{j2, \myc}(\theta)+R_{j2}^2(\theta)\big)+ e_{j1}(\theta) R_{j2, \mys}(\theta),
\end{align}
where
\begin{align}\label{equ:Rj2estA}
  \|R_{jk}^2\|_{L^\infty}= O\Big(\frac{1}{j^2}\Big),\quad\text{as}\;j \to\infty.
\end{align}
Strictly speaking, the remainder term does not satisfy Assumption \eqref{eigenfunction property} in Definition \ref{def:pp}, as we do not know the estimate for the derivatives of $R_{jk}^2$. That is, $L_{\A,a }$ is a proper perturbation of $L_{\A, 0}$ with a negligible term.

However, this term has enough decay for us to use the much weaker triangle inequality instead of the discrete multiplier theorem to get the $L^p$-boundedness. Here are the details: we only need to handle the term induced by $R_{jk}^2$. For this term, we have
\begin{align*}
   & \bigg\|\sum_{\substack{j\ge \ell                                              \\ k=1,2}} W_{jk}^rf_{jk}(r)R_{jk}^2(\theta) e_{jk}(\theta) \bigg\|_{L^p(\R^{2})}
  \lesssim \sum_{j,k} \frac1{j^2} \left\|W_{jk}^rf_{jk}(r) \right\|_{L^p(r\;dr)}      \\
   & \lesssim \sup_{j\geq\ell,k} \left\|W_{jk}^rf_{jk}(r) \right\|_{L^p(r\;dr)}
  \lesssim \bigg\|\sum_{j,k} W_{jk}^rf_{jk}(r)e_{jk}(\theta) \bigg\|_{L^p(\R^{2})} \\
   & \lesssim \bigg\|\sum_{j,k} f_{jk}(r)e_{jk}(\theta) \bigg\|_{L^p(\R^{2})}
  \lesssim \|f\|_{L^p(\R^2)}.
\end{align*}
Therefore, Theorem \ref{results in dimension two} is proved when $\tilde{A}\in\Z$ and $a(\theta)$ is symmetric across $\theta=\pi$.

\subsection{{When $\tilde{A}\in\frac12\Z\setminus\Z$ and $a$ is symmetric across $\pi$}}

In this case, the operator $L_{\A, 0}$ has the following eigenbasis
\begin{gather*}
  \mu_{j1}^2=\mu_{j2}^2 =\Big(j+\frac{1}{2}\Big)^2,
  \ j\geq 0                                              \\
  e_{j1}
  =\frac{1}{\sqrt{2\pi}}
  e^{i\int_0^\theta A(s)\dd s}
  \sin\Big(\Big(j+\frac{1}{2}\Big) \theta\Big),
  \quad e_{j2}=\frac{1}{\sqrt{2\pi}}e^{i\int_0^\theta A(s)\dd s} \cos\Big(\Big(j+\frac{1}{2}\Big) \theta\Big).
\end{gather*}
To state the properties of the eigenfunctions for the perturbed operator, we write
\begin{align*}
  \hat a(\theta)=\begin{cases}
                   4a(2\theta),      & \text{if}\  \theta \in [0, \pi],   \\
                   4a(2\theta-2\pi), & \text{if}\ \theta \in [\pi, 2\pi].
                 \end{cases}
\end{align*}
Let $\hat a_{\myc, k}$ denote the $k^\text{th}$ cosine-Fourier coefficient of $\hat a(x)$ and write
\begin{align*}
  \hat R_{jk,\mys}(\theta)= & \frac{1}{2}\sum_{m\neq 2j+1}\frac{\hat a_{\myc,|2j+1-m|}+(-1)^k \hat a_{\myc,2j+1+m}}{(2j+1-m)(2j+1+m)}\sin((m-2j-1)\theta/2), \\
  \hat R_{jk,\myc}(\theta)= & \frac{1}{2}\sum_{m\neq 2j+1}\frac{\hat a_{\myc,|2j+1-m|}+(-1)^k \hat a_{\myc,2j+1+m}}{(2j+1-m)(2j+1+m)}\cos((m-2j-1)\theta/2).
\end{align*}
Similarly as before,  the following estimates are valid,
\begin{equation}\label{equ:Cj120-again}
  \sup_{\theta}\{|\hat R_{jk,\mys}(\theta)|, |\hat R_{jk,\mys}(\theta)| |\hat R_{jk,\mys}'(\theta)|, |\hat R_{jk,\mys}'(\theta)|\}\lesssim \frac{1}{j}.
\end{equation}

For the perturbed operator $L_{\A,a}=(-i\partial_\theta+A(\theta))^2+a(\theta)$,  from the proof of \cite[Lemma B.10]{FFFParxiv}, and  \cite[Lemma B.7]{FFFParxiv}, we deduce that there exists $\ell\in\N$ such that $\{\lambda_n:n\in\N,\;n>2\ell\}=\{\nu_{j, k}^2:j\in\Z,\;j \geq\ell,\; k=1,2\}$  and
\begin{equation}\label{equ:lambdjdefAa}
  \nu_{j, k}^2=\tilde{a}+\Big(j+\frac12\Big)^2+O\Big(\frac1{j }\Big),\quad\text{as}\quad j \to\infty.
\end{equation}
The corresponding eigenfunctions $\phi_{j, k}$ satisfy
\begin{align*}
  \phi_{j1}(\theta)=e_{j1}(\theta)\big(1+\hat R_{j1, c}^1(\theta)+\hat R_{j1}^2(\theta)\big)+ e_{j2}(\theta) \hat R_{j1, s}(\theta), \\
  \phi_{j2}(\theta)=e_{j2}(\theta)\big(1+\hat R_{j2, c}^1(\theta)+\hat R_{j2}^2(\theta)\big)+ e_{j1}(\theta)\hat R_{j2, s}(\theta),
\end{align*}
where
\begin{align*}
  \|\hat{R}_{jk}^2\|_{L^\infty}=  O\Big(\frac{1}{j^2}\Big),\quad\text{as}\;j \to\infty.
\end{align*}
As before, we have that $W$ and $W^*$ are bounded in $L^p$. This concludes the final case and the proof of Theorem \ref{results in dimension two} is complete.\hfill $\square$

\appendix

\section{Proof of Proposition \ref{prop:keypropMSZ}}
\label{appendix-proof}

For any function $f(x)$ with the expansion $$f(r, \theta)=\sum_{j,k} f_{jk}(r)e_{jk}(\theta).$$
Define the operator $W_\ell$ by
\begin{equation}\label{equ:Welloper}
  W_\ell f(r,\theta)=\sum_{j = \ell}^{\infty}\sum_{k=1}^{m_j}\hankel_{\tilde{\nu}_{jk}}\hankel_{\tilde{\mu}_{jk}}f_{jk}(r) e_{jk}(\theta).
\end{equation}
Then, we rewrite $W_\ell$ as
\begin{align*}
  W_\ell f(r,\theta) & =\sum_{j,k}\int_0^\infty f_{jk}(s)\left( \int_0^\infty J_{\tilde{\nu}_{jk}} (r\lambda)J_{\tilde{\mu}_{jk}} (s\lambda ) \lambda^{d-1}\dd {\lambda}\right) s\dd s  e_{jk}(\theta) \\
                     & =\sum_{j,k} \int_0^\infty f_{jk}(s) K_{jk}(r,s)\frac{ds}{s}e_{jk}(\theta),
\end{align*}
with the kernel
$$K_{jk}(r,s)=s^d \int_0^\infty J_{\tilde{\nu}_{jk}} (r\lambda)J_{\tilde{\mu}_{jk}} (s\lambda ) \lambda^{d-1}\dd {\lambda}.$$
Hence,  \eqref{equ:MSZestkey} is equivalent to
\begin{equation}\label{equ:MSZellcases}
  \|W_\ell f\|_{L^p(\R^d)}\leq C\|f\|_{L^p(\R^d)},
\end{equation}
which can also be written as
\begin{equation}\label{equ:MSZellcases123}
  \|r^{\frac{d}{p}} W_\ell f\|_{L^p(\frac{\dd r}{r} \cdot \dd \theta)}\leq C\|r^{\frac{d}{p}} f\|_{L^p(\frac{\dd r}{r} \cdot \dd \theta)}.
\end{equation}
So we define the modified kernels
\begin{equation}\label{equ:modkerKj}
  \tilde{K}_{jk}(r,s)=\left(\frac{s}{r}\right)^{-\frac{d}{p}}K_{jk}(r,s)=\left(\frac{s}{r}\right)^{-\frac{d}{p}}s^d \int_0^\infty J_{\tilde{\nu}_{jk}} (r\lambda)J_{\tilde{\mu}_{jk}} (s\lambda ) \lambda^{d-1}\dd {\lambda},
\end{equation}
and the corresponding modified wave operator
\begin{equation}\label{equ:modwavdef}
  \tilde{W}_\ell f(r,\theta)=\sum_{j =\ell}^{\infty}\sum_{k=1}^{m_j}\int_0^\infty \tilde{K}_{jk}(r,s)f_{jk}(s)\frac{\dd s}{s}e_{jk}(\theta).
\end{equation}
Then,  \eqref{equ:MSZellcases123} can be reduced to prove
\begin{equation}\label{equ:furedmodw}
  \big\|\tilde{W}_\ell f(r,\theta)\big\|_{L^p(\frac{\dd r}{r} \cdot \dd \theta)}\leq C\|f(r, \theta)\|_{L^p(\frac{\dd r}{r} \cdot \dd \theta)}.
\end{equation}
Let
\begin{align}
  a_{jk} & =\frac{\tilde{\mu}_{jk}+\tilde{\nu}_{jk}}{2},\; b_{jk}=\frac{\tilde{\mu}_{jk}-\tilde{\nu}_{jk}}{2}.
  \label{equ:mujnujdef}
\end{align}
Since $L_{\A,0}$ and $L_{\A,a}$ satisfies the Assumption \ref{eigenvalue assumption}, and $\mathcal L_{\A, a}$ is a proper perturbation of $\mathcal L_{\A, 0}$, one has
\begin{align}\label{equ:ajbjderest}
  |a_{jk}|\simeq j ,\quad | b_{jk}|\leq \frac1{j }, \quad a_{jk} b_{jk}=O(1), \ \text{for}\ j \geq\ell.
\end{align}

By \cite[(4.9)]{MSZ}, we get the
explicit formula for the kernel functions
$\tilde K_{jk}(r,  s)$  as follows:
\begin{align}\nonumber
  \tilde K_{jk}(r,  s)= &\frac{s^{\frac{d}{2}+1-\frac{d}{p}}}{r^{\frac{d}{2}-1-\frac{d}{p}}}\int_0^\infty  \lambda J_{\tilde{\mu}_{jk}}(s\lambda) J_{\tilde{\nu}_{jk}}(r \lambda) \dd {\lambda} \\\label{equ:repKkrs}
  &=\begin{cases}
    2\left(\frac{s}{r}\right)^{\frac{d}{2}+1-\frac{d}{p}+\tilde{\mu}_{jk}} \frac{\sin(-\pi b_{jk})}{\pi}  \sum_{n=0}^\infty A_{j,k, n}^{+}\left(\frac{s}{r}\right)^{2n},   & 0<s<r; \\
    2\left(\frac{r}{s}\right)^{{\frac{d}{p}-\frac{d}{2}+1+\tilde{\nu}_{jk}}} \frac{\sin(\pi b_{jk})}{\pi}   \sum_{n=0}^\infty A_{j,k, n}^{-}\left(\frac{r}{s}\right)^{2n}, & 0<r<s,
  \end{cases}
\end{align}
with
\begin{align*}
  A_{j,k, n}^{+} & =\frac{\Gamma(a_{jk}+n+1)\Gamma(b_{jk}+n+1)}{\Gamma(a_{jk}+b_{jk}+n+1) \Gamma(n+1)},    \\
  A_{j,k, n}^{-} & = \frac{\Gamma(a_{jk}+n+1)\Gamma(1- b_{jk}+n)}{\Gamma(a_{jk}- b_{jk}+n+1) \Gamma(n+1)}.
\end{align*}

The asymptotic behavior of $\tilde K_{jk}(r,s)$ suggests that we should split $\tilde W_\ell$ into the three parts as follows:
\begin{align*}
  \tilde W_\ell f(r, \theta) & =\int_{0}^{\infty} \sum_{j,k}  \tilde{K}_j(r,s) f_{jk}(s) \frac{\dd s}{s}e_{jk}(\theta)                                                                                                      \\
                             & =\int_{0}^{\frac{r}{2}} \sum_{j,k}  \tilde{K}_j(r,s) f_{jk}(s)  \frac{\dd s}{s}e_{jk}(\theta)+\int_{\frac{r}{2}}^{2r} \sum_{j,k}  \tilde{K}_j(r,s) f_{jk}(s)   \frac{\dd s}{s}e_{jk}(\theta) \\
                             & \quad +\int_{2r}^{\infty} \sum_{j,k}  \tilde{K}_j(r,s) f_{jk}(s)  \frac{\dd s}{s}e_{jk}(\theta)                                                                                              \\
                             & =:\tilde W_{\ell,1} f(r, \theta)+\tilde W_{\ell,2} f(r, \theta)+\tilde W_{\ell,3} f(r, \theta).
\end{align*}
Then, showing \eqref{equ:furedmodw} is reduced to proving that
\begin{align}\label{equ:wellk123}
  \| \tilde W_{\ell,k} f(r, \theta)\|_{L^p(\frac{\dd r}{r} \cdot \dd \theta)} \lesssim \| f(r, \theta)\|_{L^p(\frac{\dd r}{r} \cdot \dd \theta)}, \ \text{for}\  k= 1, 2, 3.
\end{align}
We define $\tilde K_{j,k, 1}(r,s)$, $ \tilde K_{j,k, 2}(r,s)$ and $\tilde K_{j,k, 3}(r,s)$ by
\begin{align*}
  \tilde K_{j,k, 1}(r,s) & = \tilde K_{jk}(r,s) \chi_{\{0<s<r/2\}},  \\
  \tilde K_{j,k, 2}(r,s) & = \tilde K_{jk}(r,s) \chi_{\{r/2<s<2r\}}, \\
  \tilde K_{j,k, 3}(r,s) & = \tilde K_{jk}(r,s) \chi_{\{2r<s\}}.
\end{align*}


\subsection*{{Case 1: $0<\frac{s}{r}<\frac12$}.} Recall
$$\tilde K_{j,k, 1}(r,s)=2\left(\frac{s}{r}\right)^{\frac{d}{2}+1-\frac{d}{p}+\tilde{\mu}_{jk}} \frac{\sin(-\pi b_{jk})}{\pi}  \sum_{n=0}^\infty A_{j,k, n}^{+}\left(\frac{s}{r}\right)^{2n}\chi_{\{0<s<r/2\}}$$
and
$$A_{j,k, n}^{+}=\frac{\Gamma(a_{jk}+n+1)\Gamma(b_{jk}+n+1)}{\Gamma(a_{jk}+b_{jk}+n+1) \Gamma(n+1)}.$$
By \eqref{equ:ajbjderest} , and using Lemma \ref{lem:behab} with $a=a_{jk}\simeq j $ and $b= b_{jk},$ we deduce that
\begin{equation}\label{equ:Ajn+est}
  \big|A_{j,k, n}^{+}\big|\leq Cj .
\end{equation}
And so for $1\leq k \leq m_j$, by \eqref{equ:munudef}, one has
\begin{align}\label{Case-1 bd}
  |\tilde K_{j,k, 1}(r,s)|\lesssim  \left(\frac{s}{r}\right)^{\frac{d}{2}+1-\frac{d}{p}+\tilde{\mu}_{jk}} j
  \lesssim  \left(\frac{s}{r}\right)^{d -\frac{d}{p}} \frac{j }{2^{|Cj|}}\lesssim  \left(\frac{s}{r}\right)^{d -\frac{d}{p}},
\end{align}
where the above constant is independent of $j$. Hence, for the average $\overline{\tilde K}_{j, 1}(r,s)$ of $\tilde K_{j,k, 1}(r,s)$, one has
\begin{align*}
  \sup_{j}|\overline{\tilde K}_{j, 1}(r,s)| \lesssim \left(\frac{s}{r}\right)^{d -\frac{d}{p}}.
\end{align*}
On the other hand,
\begin{align}\label{Case-1 fd}
  \sup _{L\in\N:\;2^L\geq \ell} 2^{(N-1)L}  \sum_{j =2^L}^{2 ^{L+1}}\left|\fd^N\overline{\tilde K}_{j, 1}(r,s)\right|
   & \lesssim  \sum_{j =\ell}^{\infty} j ^{N-1}\sup_{1\leq k \leq m_j}| \tilde K_{j,k, 1}(r,s)| \nonumber                                        \\
   & \lesssim \sum_{j =\ell}^{\infty} j ^N \sup_{1\leq k \leq m_j}\left(\frac{s}{r}\right)^{\frac{d}{2}+1-\frac{d}{p}+\tilde{\mu}_{jk}}\nonumber \\
   & \lesssim \left(\frac{s}{r}\right)^{d-\frac{d}{p}}.
\end{align}
This inequality together with
\eqref{Case-1 bd} shows  that $\tilde K_{j,k, 1}(r,s)$ satisfies the conditions \eqref{uniform bound} and \eqref{uniform difference bound}. Since $\{e_{jk}(\theta)\}_j \in \DMB$, then for $1<p<\infty$, we have
\begin{align*}
  \left \|\sum_{j,k}\tilde K_{j,k, 1}(r,s) f_{jk}(s) e_{jk}(\theta) \right \|_{L^{p}(\dd \theta)}  \lesssim   \left(\frac{s}{r}\right)^{d-\frac{d}{p}}\left \|\sum_{j,k} f_{jk}(s) e_{jk}(\theta)\right \|_{L^{p}(\dd \theta)}.
\end{align*}
Furthermore, notice $r ^{d-\frac{d}{p}}\chi_{\{0<r<1/2\}} \in L^1(\frac{\dd r}{r})$ for all $1<p<\infty$. Combining this fact and log-Young's inequality, we have
\begin{align}\label{W 1}
  \|\tilde W_{\ell,1} f(r, \theta )\|_{ L^p (\frac{\dd r}{r} \cdot \dd \theta)}
   & \lesssim  \left\|\int_0^{r/2}  \left(\frac{s}{r}^{}\right)^{d-\frac{d}{p}} \left \|\sum_{j,k} f_{jk}(s) e_{jk}(\theta)\right \|_{L^{p}(\dd \theta)} \frac{\dd s}{s} \right \|_{L^p (\frac{\dd r}{r})}\nonumber \\
   & \lesssim \left \|\sum_{j,k}  f_{jk}(s)e_{jk}(\theta)\right \|_{L^p(\frac{\dd r}{r} \cdot \dd \theta)} =\| f(r,\theta)\|_{L^{p}(\frac{\dd r}{r} \cdot \dd \theta)}.
\end{align}

\subsection*{Case 2: $\frac{s}{r}>2$.}
Recall
$$\tilde K_{j,k, 3}(r,s)=2\left(\frac{r}{s}\right)^{{\frac{d}{p}-\frac{d}{2}+1+\tilde{\nu}_{jk}}} \frac{\sin(\pi b_{jk})}{\pi}   \sum_{n=0}^\infty A_{j,k, n}^{-}\left(\frac{r}{s}\right)^{2n}\chi_{\{0<r<s/2\}}$$
and
$$A_{j,k, n}^{-}= \frac{\Gamma(a_{jk}+n+1)\Gamma(1- b_{jk}+n)}{\Gamma(a_{jk}- b_{jk}+n+1) \Gamma(n+1)}.$$
By \eqref{equ:ajbjderest}, and using Lemma \ref{lem:behab} with $a=a_{jk}\simeq j $ and $b=- b_{jk},$ we obtain
$$\big|A_{j,k, n}^{-}\big|\leq Cj .$$
And so
\begin{align}\label{case-3-bd}
  |\tilde K_{j,k, 3}( r, s)|\lesssim \left(\frac{r}{s}\right)^{\frac{d}{p}+Cj} j  \lesssim \left(\frac{r}{s}\right)^{\frac{d}{p}}\frac{j }{2^{Cj }} \lesssim \left(\frac{r}{s}\right)^{\frac{d}{p}},
\end{align}
and
\begin{align}
  \sup _{L\in\N: 2^L\geq\ell} 2^{(N-1)L}\sum_{j=2^L}^{2 ^{L+1}}\left|\fd^N \overline{\tilde K}_{j, 3}(r,s)\right|
   & \lesssim   \sum_{j =\ell}^{\infty} j^{N-1}\sup_{1\leq k \leq m_j} | \tilde K_{j,k, 3}(r,s)|
  \nonumber                                                                                                                                                 \\
   & \lesssim  \sum_{j =\ell}^{\infty} j^N \left(\frac{r}{s}\right)^{\frac{d}{p}+Cj}  \lesssim  \left(\frac{r}{s}\right)^{\frac{d}{p}}.  	\label{case-3-fd}
\end{align}
From \eqref{case-3-bd} and \eqref{case-3-fd}, we know that
$\tilde K_{j,k, 3}(r,s)$ satisfies the condition \eqref{uniform bound} and \eqref{uniform difference bound}, which implies the following bound is valid:
\begin{align*}
  \left \|\sum_{j,k} \tilde K_{j,k, 3}(r,s) f_{jk}(s) e_{jk}(\theta) \right \|_{L^{p}(\dd \theta)}  \lesssim   \left(\frac{r}{s}\right)^{\frac{d}{p}}\left \|\sum_{j,k} f_{jk}(s) e_{jk}(\theta) \right \|_{L^{p}(\dd \theta)}.
\end{align*}

Noting that    $r^{\frac{d}{p}}\chi_{\{0<r<1/2\}}\in L^1(\frac{\dd r}{r})$ for $1<p<\infty$, we get
\begin{align}\label{W 3}
  \lVert \tilde W_{\ell,3} f(r, \theta )\rVert_{ L^p (\frac{\dd r}{r}\cdot \dd \theta)}
  & \lesssim
  \Bigg\lVert\int_{2r}^{\infty}
  \left(\frac{r}{s}\right)^{\frac{d}{p}}
  \bigg\lVert
  \sum_{j,k}
  f_{jk}(s) e_{jk}(\theta)
  \bigg \rVert_{L^{p}(\dd \theta)}
  \frac{\dd s}{s}
  \Bigg \rVert_{L^p (\frac{\dd r}{r})}
  \nonumber                                                                                                                                                         \\
   & \lesssim \bigg \|\sum_{j,k} f_{jk}(s)e_{jk}(\theta) \bigg \|_{L^p (\frac{\dd r}{r}\cdot \dd \theta)}=\| f(r,\theta)\|_{L^p (\frac{\dd r}{r}\cdot \dd \theta)}.
\end{align}
Hence, ${\tilde W}_{\ell,3}$ is bounded in  $ L^p (\frac{\dd r}{r} \cdot \dd \theta)$. And so, \eqref{equ:wellk123} with $k=3$ follows.

For the rest part of this section, we remain to show that ${\tilde W}_{\ell, 2}$ is bounded in  $ L^p (\frac{\dd r}{r} \cdot \dd \theta)$.

\subsection*{Case 3: $\frac{1}{2}<\frac{s}{r}<2$}
Now we deal with the term $\tilde W_{\ell,2} f(r, \theta )$. Here we use similar strategy as in \cite{MSZ} and we define an approximate kernel $\tilde k^{\text{ap}}(r,s)$ with the property that the index $k$ and variables $r, s $ are separated, and that they  also have the same singularity as $\tilde K_{jk}(r,s)$. For the approximating operator, we use the oscillatory in the radial part first.  After that, we use the multiplier theorem for the eigenfunctions to get the boundedness for the angular part. For the error term, we use the oscillatory in the angular part first, then using Young's inequality to control the radial part.

In the remainder of this section, we use the shorthand notation $\chi_{+}=\chi_{\{r/2<s<r\}}$ and $\chi_{-}=\chi_{\{r<s<2r\}}$. Sometimes we omit $\chi_{+}$ and $\chi_{-}$ when we have already restricted the range of $r,s$ appropriately.

Now we are in position to prove the following inequality
\begin{align}\label{equ:W2frtheest}
  \| \tilde W_{\ell,2} f(r, \theta)\|_{L^p(\frac{\dd r}{r} \cdot \dd \theta)} \lesssim   \| f(r, \theta)\|_{L^p(\frac{\dd r}{r} \cdot \dd \theta)}.
\end{align}

Using Stirling's formula, we have for any fixed $j$,  as $n$ goes to $\infty$,
\begin{align}
  A_{j,k, n}^{+} & =1-\frac{a_{jk} b_{jk}}{n+1}+O_j\left(\frac{1}{(n+1)^2}\right),\label{expansion for $A_{k,n}^{+}$}   \\
  A_{j,k, n}^{-} & =1+\frac{a_{jk} b_{jk}}{(n+1)}+O_j\left(\frac{1}{(n+1)^2}\right).\label{expansion for $A_{k,n}^{-}$}
\end{align}
This inspired us to  rewrite $\tilde K_{j,k, 2}$ as
\begin{align}\nonumber
  \tilde K_{j,k, 2} & = 2\left(\frac{s}{r}\right)^{\frac{d}{2}+1-\frac{d}{p}+\tilde{\mu}_{jk}} \frac{\sin(-\pi b_{jk})}{\pi} \left[\frac{1}{1-(\frac{s}{r})^2}-a_{jk} b_{jk}\left(\frac{s}{r}\right)^{-2} \ln\!\Big({1-\Big(\frac{s}{r}\Big)^2}\Big) +E_{jk}^{+}\right]\chi_+
  \nonumber                                                                                                                                                                                                                                                                   \\\nonumber
   & \quad +2\left(\frac{r}{s}\right)^{{\frac{d}{p}-\frac{d}{2}+1+\tilde{\nu}_{jk}}} \frac{\sin(\pi b_{jk})}{\pi} \left[\frac{1}{1-(\frac{r}{s})^2} +a_{jk} b_{jk}\left(\frac{r}{s}\right)^{-2} \ln\!\Big({1-\Big(\frac{r}{s}\Big)^2}\Big) +E_{jk}^{-} \right]\chi_-
  \\\label{equ:kjtilde2term}
  & =: \tilde K_{j,k, 2}^1 (r,s)+\tilde K_{j,k, 2}^2 (r,s)+\tilde K_{j,k, 2}^3 (r,s)
\end{align}
with
\begin{align*}
  \tilde K_{j,k, 2}^1(r,s)  & =2\left(\frac{s}{r}\right)^{\frac{d}{2}+1-\frac{d}{p}+\tilde{\mu}_{jk}} \frac{\sin(-\pi b_{jk})}{\pi}\frac{1}{1-\left(\frac{s}{r}\right)^2}\chi_+                              \\
                            & \quad +2\left(\frac{r}{s}\right)^{{\frac{d}{p}-\frac{d}{2}+1+\tilde{\nu}_{jk}}} \frac{\sin(\pi b_{jk})}{\pi}\frac{1}{1-(\frac{r}{s})^2}\chi_-;                                 \\
  \tilde K_{j,k, 2}^2 (r,s) & =-2a_{jk} b_{jk} \left(\frac{s}{r}\right)^{\frac{d}{2}-1-\frac{d}{p}+\tilde{\mu}_{jk}} \frac{\sin(- \pi  b_{jk})}{\pi}\ln\!\Big({1-\Big(\frac{s}{r}\Big)^2}\Big) \chi_+        \\
                            & \quad +2a_{jk} b_{jk} \left(\frac{r}{s}\right)^{{\frac{d}{p}-\frac{d}{2}-1+\tilde{\nu}_{jk}}} \frac{\sin (\pi  b_{jk})}{\pi}\ln\!\Big({1-\Big(\frac{r}{s}\Big)^2}\Big) \chi_-; \\
  \tilde K_{j, k,2}^3 (r,s) & =2 \left(\frac{s}{r}\right)^{\!\frac{d}{2}+1-\frac{d}{p}+\tilde{\mu}_{jk}} \frac{\sin(- \pi  b_{jk})}{\pi} E_{jk}^+(r,s)                                                       \\
                            & \quad +2 \left(\frac{r}{s}\right)^{\!{\frac{d}{p}-\frac{d}{2}+1+\tilde{\nu}_{jk}}} \frac{\sin (\pi  b_{jk})}{\pi}E_{jk}^-(r,s).
\end{align*}
We denote the corresponding operators $\tilde W_{\ell,2}^1, \tilde W_{\ell,2}^2$ and $\tilde W_{\ell,2}^3$ by
\begin{align}
  \tilde W_{\ell,2}^{i} f(r, \theta)=\int_0^\infty \sum_{j,k}\tilde K_{j,k, 2}^i(r,s)  f_{jk}(s)e_{jk}(\theta)\frac{\dd s}{s}, \qquad i=1, 2, 3.\notag
\end{align}
\vspace{0.2in}
\noindent\emph{\textbf{Estimate for the first term $\tilde W_{\ell,2}^{1}$.}}

Step 1: \emph{Estimate for the approximate kernels: }Let us define the approximate kernels by
\begin{align*}
  \tilde K_{j,k, 2}^{1, \text{ap}}(r,s) & =  2 \left(\frac{s}{r}\right)^{\frac{d}{2}+1-\frac{d}{p}+\tilde{\mu}_{\ell, 1}} \frac{\sin (-\pi  b_{jk})}{\pi}\frac{1}{1-\left(\frac{s}{r}\right)^2}\chi_+
  \\ &\quad + 2 \left(\frac{r}{s}\right)^{\frac{d}{p}-\frac{d}{2}+1+\tilde{\nu}_{\ell, 1}} \frac{\sin \pi  b_{jk}}{\pi}\frac{1}{1-(\frac{r}{s})^2}\chi_-,
\end{align*}
and  the corresponding operator $\tilde W_{\ell,2}^{1, \text{ap}}f (r, \theta)$   by
\begin{align*}
  \tilde W_{\ell,2}^{1, \text{ap}}f(r,\theta)\  & \coloneq \sum_{j,k} \int_{0}^\infty \tilde K_{j,k, 2}^{1, \text{ap}}(r,s) f_{jk}(s)\frac{\dd s}{s}e_{jk}(\theta)                                                                                                \\
                                                & =2\int_\frac{r}{2}^r  \left(\frac{s}{r}\right)^{\frac{d}{2}+1-\frac{d}{p}+\tilde{\mu}_{\ell, 1}}\frac{1}{1-(\frac{s}{r})^2}  \sum_{j,k} \frac{\sin (-\pi  b_{jk})}{\pi}  f_{jk}(s)e_{jk}(\theta)\frac{\dd s}{s} \\
                                                & \quad +2\int_{r}^{2r}  \left(\frac{r}{s}\right)^{\frac{d}{p}-\frac{d}{2}+1+\tilde{\nu}_{\ell, 1}}\frac{1}{1-(\frac{r}{s})^2}  \sum_{j,k}\frac{\sin \pi  b_{jk}}{\pi} f_{jk}(s)e_{jk}(\theta)\frac{\dd s}{s}.
\end{align*}
Using  Lemma \ref{lem:CZbd}, we have
\begin{align*}
  \|\tilde W_{\ell,2}^{1, \text{ap}}f(r,\theta)\|_{L^p(\frac{dr}{r})}\lesssim  \bigg\|\sum_{j,k} \sin \pi  b_{jk} f_{jk}(r)e_{jk}(\theta)\bigg\|_{L^p(\frac{dr}{r})}.
\end{align*}
We can easily check that the sequence $\{\sin \pi  b_{jk}\}$ satisfies the condition   \eqref{uniform bound} and \eqref{uniform difference bound}. Hence
\begin{align}\label{equ:wellappro}
  \|\tilde W_{\ell,2}^{1, \text{ap}}f(r,\theta)\|_{L^p(\frac{dr}{r}\cdot d\theta)}\lesssim  \bigg\|\sum_{j,k} f_{jk}(r)e_{jk}(\theta)\bigg\|_{L^p(\frac{dr}{r}\cdot d\theta)}.
\end{align}

Step 2:\noindent\emph{ Estimate for the error term}. Define the kernels of the error terms
\begin{align*}
  \tilde K_{j,k, 2}^{1, \text{err}}(r,s) & =\tilde K_{j,k, 2}^{1}(r,s)-\tilde K_{j,k, 2}^{1, \text{ap}}(r,s)                                                                                                                                                                         \\
                                         & =2\left(\frac{s}{r}\right)^{\frac{d}{2}+1-\frac{d}{p}+\tilde{\mu}_{\ell, 1}} \frac{\sin(-\pi b_{jk})}{\pi}\frac{1}{1-\left(\frac{s}{r}\right)^2}\left[\left(\frac{s}{r}\right)^{{\tilde{\mu}_{jk}-\tilde{\mu}_{\ell, 1}}}-1\right] \chi_+ \\
                                         & \quad +2\left(\frac{r}{s}\right)^{\frac{d}{p}-\frac{d}{2}+1+\tilde{\nu}_{\ell, 1}} \frac{\sin(\pi b_{jk})}{\pi}\frac{1}{1-(\frac{r}{s})^2}\left[\left(\frac{r}{s}\right)^{{\tilde{\nu}_{jk}-\tilde{\nu}_{\ell, 1}}}-1\right] \chi_-,
\end{align*}
and the associated operator
\begin{align*}
  \tilde W_{\ell,2}^{1, \text{err}} f(r, \theta)=\int_0^\infty \sum_{j,k}\tilde K_{j,k, 2}^{1, \text{err}}(r,s)  f_{jk}(s)e_{jk} (\theta)\frac{\dd s}{s}.
\end{align*}
By the same argument as in \cite[Lemma 4.3, Lemma 4.5]{MSZ},  we have
\begin{align*}
  | \overline{\tilde K}_{j 2}^{1, \text{err}}(r,s)|\lesssim 1, \quad \sup_{L\in\N:2^L\geq\ell} 2^{(N-1)L} \sum_{j=2^L}^{2^{L+1}}|\fd ^N \overline{\tilde K}_{j, 2}^{1, \text{err}}(r,s)|\lesssim 1.
\end{align*}

Hence
\begin{align*}
  \| \tilde W_{\ell,2}^{1, \text{err}} f(r, \theta) \|_{L^p(\dd \theta)} \leq \int_{r/2}^r  \Big\|\sum_{j,k}f_{jk}(s)e_{jk}(\theta)\Big\|_{L^p(\dd \theta)}\frac{\dd s}{s}.
\end{align*}
Furthermore, we obtain
\begin{align}\label{equ:lpwellerr2}
  \| \tilde W_{\ell,2}^{1, \text{err}} f(r, \theta)\|_{L^p(\frac{dr}{r}\cdot \dd\theta)}  \leq \|f(r, \theta)\|_{L^p(\frac{dr}{r}\cdot \dd\theta)}.
\end{align}
Combining this with \eqref{equ:wellappro}, we get the $L^p$-boundedness of $\tilde W_{\ell,2}^1$.

\vspace{0.2in}
{{\noindent\emph\textbf{Estimate for the second term $\tilde W_{\ell,2}^{2} $.}} Direct computation gives
\begin{align*}
  | \overline{\tilde K_{j, 2}} (r,s)|                                                  & \lesssim \left|\ln\!\Big({1-\Big(\frac{s}{r}\Big)^2}\Big)\right|\chi_++\left|\ln\!\Big({1-\Big(\frac{r}{s}\Big)^2}\Big)\right|\chi_-, \\
  \sup_L  2^{(N-1)L}   \sum_{j=2^L}^{2^{L+1}}|\fd^N \overline{\tilde K_{j, 2}^2}(r,s)| & \lesssim
  \left|\ln\!\Big({1-\Big(\frac{s}{r}\Big)^2}\Big)\right|\chi_++\left|\ln\!\Big({1-\Big(\frac{r}{s}\Big)^2}\Big)\right|\chi_-.
\end{align*}
Since $\{e_\alpha\}_{\alpha\in \II \cup \IL} \in \DMB$, we have
\begin{align*}
  \|\tilde W_{\ell,2}^{2}  f(r, \theta) \|_{L^p(\dd \theta)}
   & \leq \int_{r/2}^r
  \left|
  \ln\!\Big({1-\Big(\frac{s}{r}\Big)^2}\Big)
  \right|
  \bigg\lVert
  \sum_{j,k}
  f_{jk}(s)
  e_{jk}(\theta)
  \bigg\rVert_{L^p(\dd \theta)}
  \frac{\dd s}{s}                                                                                                                                                           \\
   & \quad +\int_r^{2r}\left|\ln\!\Big({1-\Big(\frac{r}{s}\Big)^2}\Big)\right|\bigg\lVert \sum_{j,k} f_{jk}(s)e_{jk}(\theta)  \bigg\rVert_{L^p(\dd \theta)}\frac{\dd s}{s}.
\end{align*}
Hence, by Young's inequality, we get
\begin{align}\label{equ:lpwellerr21}
  \|\tilde W_{\ell,2}^{2}  f(r, \theta) \|_{L^p(\frac{\dd r}{r} \cdot \dd{\theta})} \lesssim  \|f(r, \theta)\|_{L^p(\frac{\dd r}{r} \cdot \dd{\theta})}.
\end{align}

\vspace{0.2in}

\noindent\emph\textbf{Estimate for the third term $\tilde W_{\ell,2}^{3}$.} Recall that $\tilde K_{j,k, 2}^3$ is given by
\begin{align*}
  \tilde K_{j,k, 2}^3 (r,s)
   & = 2 \left(\frac{s}{r}\right)^{\!\frac{d}{2}+1-\frac{d}{p}+\tilde{\mu}_{jk}}\frac{\sin (-\pi  b_{jk})}{\pi}\sum_{n=0}^{\infty} E_{j,k, n}^{+}\left(\frac{s}{r}\right)^{2n}\chi_+  \\
   & \quad +2 \left(\frac{r}{s}\right)^{\frac{d}{p}-\frac{d}{2}+1+\tilde{\nu}_{jk}} \frac{\sin \pi  b_{jk}}{\pi}\sum_{n=0}^{\infty}E_{j,k,n}^{-}\left(\frac{r}{s}\right)^{2n} \chi_-,\end{align*}
where
\begin{align} \label{A equal to E}
  E_{j,k, n}^{+}= A_{j,k, n}^{+}-1+\frac{a_{jk} b_{jk}}{n+1},\  1\leq k \leq m_j ; \\
  E_{j,k, n}^{-}=  A_{j,k, n}^{-}-1-\frac{a_{jk} b_{jk}}{n+1},\    1\leq k \leq m_j.
\end{align}
Follow the same argument as in \cite[Lemma 4.6]{MSZ}, we can prove that
\begin{align}
  \bigg| \frac{\dd^N E_{j,k, n}^{+}}{\dd j^N}\bigg| \lesssim \frac{1}{j^N} \frac{1}{n+1}, \label{derivative for A plus} \\
  \bigg|\frac{\dd^N E_{j,k, n}^{-}}{\dd j^N}\bigg| \lesssim \frac{1}{j^N} \frac{1}{n+1}  \label{derivative for A minus}.
\end{align}
These will lead to that $\tilde K_{j, 2}^3 (r,s)$ will satisfies the conditions \eqref{uniform bound} and \eqref{uniform difference bound} with an upper bound $ |\ln\left(1-(\frac{s}{r})^2\right)|$ for fixed $r, s$, hence

\begin{align*}
  \|\tilde W_{\ell,2}^{3} f(r, \theta) \|_{L^p(\dd \theta)}
   & \leq
  \int_{r/2}^r
  \Big|\ln\!\Big({1-\Big(\frac{s}{r}\Big)^2}\Big)\Big|
  \Big\|\sum_{j,k}f_{jk}(s)e_{jk}(\theta)\Big\|_{L^p(\dd \theta)}\frac{\dd s}{s} \\
   & +\int_r^{2r}\Big|\ln\!\Big({1-\Big(\frac{r}{s}\Big)^2}\Big)\Big|
  \Big\|\sum_{j,k}f_{jk}(s)e_{jk}(\theta)\Big\|_{L^p(\dd \theta)}\frac{\dd s}{s}.
\end{align*}
Log-Young's inequality again gives
\begin{align}\label{W 2 third term}
  \| \tilde W_{\ell,2}^{3} f(r, \theta) \|_{L^p(\frac{\dd r}{r} \cdot \dd{\theta})} \lesssim  \|f(r, \theta)\|_{L^p(\frac{\dd r}{r} \cdot \dd{\theta})}.
\end{align}
Combining \eqref{equ:lpwellerr2}, \eqref{equ:lpwellerr21} and \eqref{W 2 third term}, we have proved that
\begin{align}\label{W 2}
  \| \tilde W_{\ell,2} f(r,\theta)\|_{L^p(\frac{\dd r}{r} \cdot \dd{\theta})} \lesssim  \|f(r, \theta)\|_{L^p(\frac{\dd r}{r} \cdot \dd{\theta})}.
\end{align}
This inequality together with
\eqref{W 1} and \eqref{W 3} implies that $\tilde W$ is bounded in \mbox{$L^p(\frac{\dd r}{r} \cdot \dd{\theta})$} for $1<p<\infty$. Hence,  we conclude the proof Proposition \ref{prop:keypropMSZ}.

\section*{Acknowledgements}
L. Fanelli and Y. Wang are partially supported by the Basque Government through the BERC 2022--2025 program and by the Spanish Agencia Estatal de Investigaci\'on through BCAM Severo Ochoa excellence accreditation CEX2021-001142-S/MCIN/AEI 10.13039/501100011033, and by the Research Project PID2021-123034NB-I00 funded by MCIN/AEI/10.13039/501100011033.

L. Fanelli is also supported by the project IT1615-22 funded by the Basque Government.

X. Su's postdoctoral position is funded by the Engineering \& Physical Sciences Research Council (EPSRC) under grant EP/X011488/1.

J. Zhang was supported by  National key R\&D program of China: 2022YFA1005700, National Natural Science Foundation of China(12171031) and Beijing Natural Science Foundation(1242011).

J. Zheng was supported by National key R\&D program of China: 2021YFA1002500 and NSF grant of China (No. 12271051).

\begin{center}
  
\end{center}


\begin{thebibliography}{99}

    \bibitem{AB}
    Y. Aharonov, and D. Bohm, Significance of electromagnetic potentials in quantum theory,
    \textit{Phys. Rev.} \textbf{115} (1959), 485-491.

    %
    %

    \bibitem{BKB} H. Bartolomei, M. Kumar, R. Bisognin, et al., Fractional statistics in anyon collisions, \textit{Science}, \textbf{368} (2020), 173--177.

    \bibitem{BC} A. Bonami, J. L. Clerc, \emph{Sommes de Ces\` aro et multiplicateurs des developpements en harmoniques sph\'eriques}, Trans. Amer. Math. Soc., \emph{183} (1973), 223--263.

    \bibitem{BG} M. Beceanu and M. Goldberg, Schr\"odinger dispersive estimates for a scaling-critical class of potentials, Comm. Math. Phys. 314 (2012), 471-481.

    \bibitem{BPSS} N. Burq, F. Planchon, J. Stalker and A. S.
    Tahvildar-Zadeh, Strichartz estimates for the wave and Schr\"odinger
    equations with the inverse-square potential, J. Funct. Anal., 203
    (2003), 519-549.

    \bibitem{BPST} N. Burq, F. Planchon, J. Stalker and A. S. Tahvildar-Zadeh, Strichartz estimates for the wave and {S}chr\"odinger equations with potentials of critical decay,
    Indiana Univ. Math. J. 53 (2004), 1665-1680.

    \bibitem{CS72} L. Carleson, P. Sj\"olin, Oscillatory integrals and a multiplier problem for the disc, Studia Math. 44
    (1972), 287-299.


    %
    %
    %





    \bibitem{Dun} T. M. Dunster, On the order of derivatives of Bessel functions, Constructive Approximation, \emph{46} (2017), 47-68.


    \bibitem{DF1} P. D'Ancona, and L. Fanelli, $L^p$-Boundedness of the Wave Operator for the One Dimensional Schr\"dinger Operator, \textit{Comm. Math. Phys.} \textbf{268} (2006), 415--438.

    %


    \bibitem{Eastham} M. S. P. Eastham, The Spectral Theory of Periodic Differential Equations, Scottish Academic
    Press, 1973

    \bibitem{EGG} M. B. Erdogan, M. Goldberg, and W. R. Green, Counterexamples to $L^p$-boundedness of wave operators for classical and higher order Schr\"odinger operators, \textit{J. Funct. Anal.} \textbf{285}, (2023) no. 5, 110008. 16 pages.

    \bibitem{EGL} M. B. Erdogan, W. Green, and K. Lamaster, $L^p$-continuity of Wave Operators for higher order Schr\"odinger operators with threshold eigenvalues in high dimensions, arXiv:2407.07069 (2024).

    %


    \bibitem{Ev}  G. Ev\'equoz, Existence and asymptotic behavior of standing waves of the nonlinear Helmholtz equation in the plane,
    \textit{Analysis} (Berlin) 37 (2017), 55--68.






    \bibitem{FFFP1}
    L. Fanelli, V. Felli, M. A. Fontelos, and A. Primo, Time decay of scaling critical electromagnetic Schr\"odinger flows,
    Comm. Math. Phys., 324(2013), 1033-1067.

    \bibitem{FFFParxiv}
    L. Fanelli, V. Felli, M. A. Fontelos, and A. Primo, Time decay of scaling invariant electromagnetic Schr\"odinger equations on the plane,
    arXiv:1405.1784 (2014)

    \bibitem{FFFP}
    L. Fanelli, V. Felli, M. A. Fontelos, and A. Primo, Time decay of scaling invariant electromagnetic Schr\"odinger equations on the plane,
    Comm. Math. Phys., 337(2015), 1515-1533.

    \bibitem{FFFP2}
    L. Fanelli, V. Felli, M. A. Fontelos, and A. Primo, Frequency-dependent time-decay pf Schr\"odinger flows, \textit{J. Spectral Theory} \textbf{8} (2018), 509--521.






    \bibitem{FGK15}  L. Fanelli, G. Grillo, H. Kovarík, Improved time-decay for a class of scaling critical
    electromagnetic Schr\"odinger flows, Journal of Functional Analysis, 269(2015), 3336--3346.


    \bibitem{FZZ} L. Fanelli, J. Zhang and J. Zheng,
    Dispersive estimates for 2D-wave equations with critical potentials, Advances in Mathematics, 400(2022), 108333.


    \bibitem{FZZ23} L. Fanelli, J. Zhang and J. Zheng,  Uniform resolvent estimates for critical magnetic Schr\"odinger operators in 2D,
    International Mathematics Research Notices, 20(2023), 17656--17703.


    \bibitem{FFT} V. Felli, A. Ferrero, S. Terracini,  Asymptotic behavior of solutions to Schr\"odinger equations near an
    isolated singularity of the electromagnetic potential. J. Eur. Math. Soc. 13(2011), 119-174.


    \bibitem{GWZZ} X. Gao, J. Wang, J. Zhang and J. Zheng, Uniform resolvent estimates for Schr\"odinger operators in Aharonov--Bohm magnetic fields,
    \textit{Journal of Diff. Equ.}, 292(2021), 70-89.

    \bibitem{GYZZ} X. Gao, Z. Yin, J. Zhang and J. Zheng, Decay and Strichartz estimates in critical
    electromagnetic fields, Journal of Functional Analysis, 282(2022), 109350.

    \bibitem{Gur} D. Gurarie, \emph{Zonal {Schr{\"o}dinger} operators on the {{\(n\)}}-sphere: {Inverse} spectral problem and rigidity}, Commun. Math. Phys., 131(1990), 571--603.

    \bibitem{Gut}
    S. Guti\'errez, \emph{Non trivial $L^q$ solutions to the Ginzburg--Landau equation}, Math. Ann., \emph{328} (2004), 1--25.

    \bibitem{G} M. Goldberg, Strichartz estimates for the Schr\"odinger equation with time-periodic $L^{n/2}$ potentials, J. Funct. Anal. 256 (2009), 718-746.


    \bibitem{JK} A. Jensen and T. Kato, Spectral properties of Schr\"odinger operators and time-decay of the wave functions, Duke Math. J. 46 (1979), 583-611.


    \bibitem{K}
    T. Kato, Perturbation theory for linear operators, Springer--Verlag,
    Berlin, 1966.



    \bibitem{KRS} C. E. Kenig, A. Ruiz, and C. D. Sogge, \emph{Uniform Sobolev inequalities and unique continuation for second order constant coefficient differential operators}, Duke Math. J., \emph{55} (1987), 329--347.


    %
    %

    \bibitem{MS} A. C. McBride, and W. J. Spratt, A class of {Mellin} multipliers, Can. Math. Bull., \emph{35} (1992), 252--260.


    \bibitem{MSZ} C. Miao, X. Su and J. Zheng, The $W^{s,p}$-boundedness of stationary wave operators for the Schr\"odinger operator with inverse-square potential, Tran. Amer. Math. Soc., 376(2023), 1739--1797.


    \bibitem{MYZ} C. Miao, L. Yan and J. Zhang, Bochner--Riesz means for critical magnetic Schr\"odinger operators in the plane.   arXiv:2405.02531.



    %
    %
    %

    \bibitem{PT89}
    M. Peshkin, and A. Tonomura. The Aharonov--Bohm Effect. Lect. Notes Phys. \textbf{340} (1989).











    \bibitem{Stark}
    {J. Stark} \textit{Beobachtungen \"uber den Effekt des elektrischen Feldes auf Spektrallinien I.
      Quereffekt}, Annalen der Physik, vol. 43, pp. 965--983 (1914)

    \bibitem{Sogge}
    {C. Sogge}, Fourier Integrals in Classical Analysis, \textit{Cambridge University Press} 1993.


    \bibitem{Stein86} E.M. Stein, Oscillatory integrals in Fourier analysis, in: Beijing Lectures in Harmonic Analysis,
    Beijing, 1984, in: Ann. of Math. Stud., vol. 112, Princeton Univ. Press, Princeton, NJ, 1986, pp.
    307-355.

    \bibitem{TW} L. Thomas, and S. Wassell, Semiclassical approzimation for Schr\"odinger operators on a two-sphere at high energy, \textit{J. Math. Phys.} \textbf{36} (1995), 5480--5505.

    \bibitem{Tricomi} F. G. Tricomi, and A. Erd\'elyi, \textit{The asymptotic expansions of a ratio of Gamma functions,} Pacific J. Math., \emph{1}, 133--142 (1951).

    \bibitem{W} R. Weder, The $W^{k,p}$-continuity of Schr\"odinger wave operators on the line, \textit{Comm. Math. Phys.} \textbf{208} (1999), 507--520.

    \bibitem{Y1} K. Yajima, The $W^{k,p}$-continuity of wave operators for Schr\"odinger operators. \textit{J. Math. Soc. Japan}
    \textbf{47} (1995), no. 3, 551--581.

    \bibitem{Y2} K. Yajima, The $W^{k,p}$-continuity of wave operators for Schr\"odinger operators. II. Positive potentials in even dimensions $m \geq 4$. Spectral and scattering theory (Sanda, 1992), 287--300, Lecture Notes in Pure and Appl. Math., 161, Dekker, New York, 1994.

    \bibitem{Y3} K. Yajima, The $W^{k,p}$-continuity of wave operators for Schr\"odinger operators. III. Even-dimensional cases $m \geq 4$. \textit{J. Math. Sci. Univ. Tokyo} \textbf{2} (1995), no. 2, 311--346.

    \bibitem{Y4} K. Yajima, The $L^{p}$-continuity of wave operators for Schr\"odinger operators with threshold singularities I. The odd dimensional case. \textit{J. Math. Sci. Univ. Tokyo} \textbf{13} (2006), 43--94.

    \bibitem{Y5} K. Yajima, Wave operators for Schr\"odinger operators with threshold singularities Revisited. Preprint, arXiv:1508.05738.

    \bibitem{Y6} K. Yajima Remark on the Lp-boundedness of wave operators for Schr\"odinger operators with threshold singularities, \textit{Documenta Mathematica} \textbf{21} (2016), 391--443.

    \bibitem{Zeeman} P. Zeeman, \textit{Over de invloed eener magnetisatie op den aard van het door een stof uitgezonden licht}, Verslagen van de Gewone Vergaderingen der Wis - en Natuurkundige Afdeeling (Koninklijk Akademie van Wetenschappen te Amsterdam) \textbf{5} (1896), 181-184; 242-248









  \end{thebibliography}
\end{document}